\documentclass{amsart}
\usepackage{amssymb,amsmath,amsfonts,mathrsfs,mathtools}
\usepackage{multirow}
\usepackage{color}
\usepackage{dsfont}
\usepackage{latexsym,bm}
\usepackage{flafter}
\usepackage{longtable}
\usepackage{changepage}
\usepackage{float}
\usepackage{epsfig}
\usepackage{epstopdf}
\usepackage{enumitem}
\usepackage{array}
\usepackage{listings}
\usepackage{graphicx,subfigure}
\usepackage{booktabs,caption}
\usepackage{parskip}
\usepackage{url}

\newtheorem{theorem}{Theorem}[section]

\newtheorem{example}[theorem]{Example}%

\newtheorem{cor}[theorem]{Corollary}

\theoremstyle{definition}

\newtheorem{algorithm}[theorem]{Algorithm}

\numberwithin{equation}{section}

\newcommand{\re}{\mathbb{R}}
\newcommand{\st}{\mathit{s.t.}}
\newcommand{\reff}[1]{(\ref{#1})}

\newcommand{\mc}[1]{\mathcal{#1}}

\newcommand{\qmod}[1]{\mathit{QM}[#1]}

\newcommand{\bdes}{\begin{description}}
	\newcommand{\edes}{\end{description}}
\newcommand{\bal}{\begin{align}}
	\newcommand{\eal}{\end{align}}
\newcommand{\bnum}{\begin{enumerate}}
	\newcommand{\enum}{\end{enumerate}}
\newcommand{\bit}{\begin{itemize}}
	\newcommand{\eit}{\end{itemize}}
\newcommand{\bea}{\begin{eqnarray}}
	\newcommand{\eea}{\end{eqnarray}}
\newcommand{\be}{\begin{equation}}
	\newcommand{\ee}{\end{equation}}
\newcommand{\baray}{\begin{array}}
	\newcommand{\earay}{\end{array}}
\newcommand{\bsry}{\begin{subarray}}
	\newcommand{\esry}{\end{subarray}}
\newcommand{\bca}{\begin{cases}}
	\newcommand{\eca}{\end{cases}}
\newcommand{\bcen}{\begin{center}}
	\newcommand{\ecen}{\end{center}}
\newcommand{\bbm}{\begin{bmatrix}}
	\newcommand{\ebm}{\end{bmatrix}}
\newcommand{\bmx}{\begin{matrix}}
	\newcommand{\emx}{\end{matrix}}
\newcommand{\bpm}{\begin{pmatrix}}
	\newcommand{\epm}{\end{pmatrix}}
\newcommand{\btab}{\begin{tabular}}
	\newcommand{\etab}{\end{tabular}}

\begin{document}
	
	\title[Robust approximation CCO]{Robust approximation of chance constrained optimization 
			with polynomial perturbation}
	
	\author[Bo Rao]{Bo Rao}
	\address{Bo Rao, School of Mathematics and Computational Sciences,
		Xiangtan University, Hunan, China}
	\email{raob@smail.xtu.edu.cn}
	
	\author[Liu Yang]{Liu Yang}
	\address{Liu Yang, Guangming Zhou, School of Mathematics and Computational Sciences, Xiangtan University, Hunan, China}
	\address{Hunan Key Laboratory for Computation and Simulation in Science and Engineering, Xiangtan University, Hunan, China}
	\email{yangl410@xtu.edu.cn, zhougm@xtu.edu.cn}
	
	\author[Guangming Zhou]{Guangming Zhou}
	
	\author[Suhan Zhong]{Suhan Zhong}
	\address{Suhan Zhong, Department of Mathematics, Texas A\&M University,
		College Station, Texas, USA, 77843}
	\email{suzhong@tamu.edu}
	
	
	
	\begin{abstract}
		This paper proposes a robust approximation method for solving chance constrained optimization (CCO) of polynomials. 
		Assume the CCO is defined with an individual chance constraint that is affine in the decision variables. 
		We construct a robust approximation by replacing the chance constraint with a robust constraint over an uncertainty set.
		When the objective function is linear or SOS-convex, 
		the robust approximation can be equivalently transformed into linear conic optimization. 
		Semidefinite  relaxation algorithms are proposed to solve these linear conic transformations globally
		and their convergent properties are studied.
		We also introduce a heuristic method to find efficient uncertainty sets 
		such that optimizers of the robust approximation are feasible to the original problem. 
		Numerical experiments are given to show the efficiency of our method.
	\end{abstract}
	
	\maketitle
	\section{Introduction}
Many real-world decision problems can be conveniently described by optimization with uncertainties. 
The chance constrained optimization (CCO) is a popular framework in stochastic programming.
It aims at finding the best decisions such that random constraints are satisfied with a probability 
greater than or equal to a specified threshold. A CCO problem is
\begin{equation}\label{eq:cco}
	\left\{\begin{array}{cl}
		\min\limits_{x \in X}  & f(x) \\
		\st &  \mathbb{P}\left\{\xi : h(x,\xi) \geq 0 \right\} \geq 1 - \epsilon,
	\end{array}
	\right.
\end{equation}
where $x\in \mathbb{R}^n$ is the decision variable constrained in a set $X$,
$\xi\in\mathbb{R}^{r}$ is the random vector with the probability distribution $\mathbb{P}$ supported in 
a given set $S\subseteq \re^r$. 
The $f:\mathbb{R}^n\to \mathbb{R}$ is a function in $x$, 
$h:\, \mathbb{R}^n\times \mathbb{R}^{r}\to\mathbb{R}^{m_0}$ is a scalar or vector-valued function, 
and $\epsilon\in(0,1)$ is a given risk level. 
In \reff{eq:cco}, we use $\mathbb{P}\left\{\xi: h(x,\xi) \geq 0 \right\}$ to denote the probability of 
$h(x,\xi)\ge 0$ with respect to the distribution $\mathbb{P}$.
The corresponding random constraint is called the {\it chance constraint}.
When $h$ is a scalar function, \reff{eq:cco} is called an \textit{individual} chance constrained problem.
Otherwise, it is called a \textit{joint} chance constrained problem if $m_0>1$. 
The (\ref{eq:cco}) is said to have a {\it polynomial perturbation} if $h(x,\xi)$ is
a polynomial scalar or vectorin $\xi$.
This class of CCO problems has been studied in \cite{FEN11,LAG05,LJB17,JAS15}.
CCO is very difficult to solve due to the complicated structure of its feasible set. 
In this paper, we propose an efficient robust approximation approach to handle individual CCO problems 
with polynomial perturbation, under some convex assumptions in the decision variables.

Chance constrained optimization has broad applications in finance \cite{BKP09}, emergency management \cite{CDE12}, 
and energy management \cite{CHA65, DBI14}. 
Individual CCO problems are studied in \cite{AHM14, DBE21, ANE07, YCA20}. 
For a comprehensive introduction to the topic, we refer to monographs \cite{PRE03, DEN21} and references therein. 
It is very challenging to solve a general CCO problem. 
A major reason is that the feasible set usually does not have a convenient characterization in computations.
Approximation approaches are often used to solve CCO problems.
Scenario approximation is given in \cite{GCC06, ANE06}. 
Sample average approximation (SAA) is given in \cite{CHE19, JLU08, BKP09}. 
Condition value at risk (CVaR) approximation is given in \cite{ANE07, ROC00}. 
Bernstein approximation is given in \cite{ANE07}. 
DC approximation is given in \cite{SFE14, LJH11}.
Smooth and nonsmooth approximations are given in \cite{RKA21, YCA20}.

In this paper, we propose a robust approximation approach for individual CCO problems with polynomial perturbation. 
Assume $f,h$ in \reff{eq:cco} are scalar polynomials and that $h$ is affine in $x$. 
Let $d$ denote the highest degree of $h(x,\xi)$ in $\xi$.
One can use a scalar matrix $A$ and vector $b$ to represent
\begin{equation}\label{eq:poly_pert} 
	h(x,\xi) \,=\, (Ax+b)^T[\xi]_d
\end{equation}
as a vector inner product, where $[\xi]_d :=(1,\xi_1,\ldots,\xi_r,\xi_1^2,\ldots,\xi_r^d)^T$
is the monomial vector in $\xi$ of the highest degree $d$ ordered alphabetically.
Since the chance constraint is nonlinear in $\xi$, the feasible set of \reff{eq:cco} is hard to characterize computationally. 
Instead, it is much simpler to find an uncertainty set $U\subseteq \re^r$ such that for every $x\in X$,
\[
h(x,\xi)\geq 0,\, \forall \xi\in U\quad \Rightarrow
\quad\mathbb{P}\left\{\xi: h(x,\xi) \geq 0 \right\} \geq 1 - \epsilon.
\]
This implies a robust approximation of (\ref{eq:cco}) as
\begin{equation}\label{eq:robust_approximation}
	\left\{\begin{array}{cl}
		\min\limits_{x \in X} & f(x) \\ 
		\st & h(x,\xi) \geq 0,\quad \forall\xi \in U.
	\end{array}\right.
\end{equation}

We refer to \cite{ABN00, DBE21, DBE04, HON21, ZLI14, YYU17} for relevant work on the robust approximation approach.
Let $\mathscr{P}(U)$ denote the set of polynomials in $\xi$ that are nonnegative on $U$.
Then we can equivalently reformulate \reff{eq:robust_approximation} into
\begin{equation}\label{eq:lc_app}
	\left\{\begin{array}{cl}
		\min\limits_{x\in X} & f(x)\\
		\st &  h(x,\xi) = (Ax+b)^T[\xi]_d \in \mathscr{P}(U).\\
	\end{array}\right.
\end{equation}
In the above, the membership constraint means $h(x,\xi)$, as a polynomial in $\xi$, belongs to $\mathscr{P}(U)$.

When $f$ is affine in $x$ and $X$ has a semidefinite representation,
the \reff{eq:lc_app} is a linear conic optimization problem involving a Cartesian product 
of semidefinite and nonnegative polynomial cones.
When $f$ is a general polynomial, \reff{eq:lc_app} can be further relaxed into linear conic optimization
by introducing new variables to represent nonlinear monomials. 
Such a relaxation is tight if $f(x)$ and $X$ are given by SOS-convex polynomials
\footnote[1]{A polynomial $f\in \mathbb{R}[x]$ is said to be SOS-convex if there exists a 
	matrix-polynomial $F(x)$ such that its Hessian matrix $\nabla^2 f=F(x)^TF(x)$.}.
Under the archimedeaness and some general assumptions, the transformed linear conic 
optimization can be solved globally by Moment-Sum-of-Squares (SOS) relaxations \cite{NIE14,NIE15}. 

For convenience, we assume the uncertainty set in \reff{eq:robust_approximation} has the form
\[
U \,=\, \{\xi\in\re^r: \Gamma -(\xi-\mu)^T\Lambda^{-1}(\xi-\mu) \,\geq\, 0 \},
\]
where $\Gamma>0$ is a scalar, $\mu\in\re^r$ is a real vector and $\Lambda\in\re^{r\times r}$ is a positive definite matrix. 
In this paper, $\mu$ and $\Lambda$ are simply chosen
as the mean value vector and covariance matrix of the random variable. 
In numerical experiments, they are either given directly or computed from samples.
The scalar $\Gamma$ is used to describe the size of $U$. The selection of $\Gamma$ can be tricky.
It should be big enough such that the optimal solution of \reff{eq:robust_approximation} 
is feasible to the original problem \reff{eq:cco}. 
Meanwhile, a relatively smaller $\Gamma$ is usually preferred for an efficient approximation.
In computations, we use a heuristic iterative method to find good $\Gamma$. 
We first determine an initial value by the quantile estimation approach,
and then update based on the probability of constraint violation of computed solutions. 
The quality of a robust approximation is usually dependent on the selection of uncertainty sets, 
see related works in \cite{GUZ16, GUZ17, GUZ17a, ZLI12, DBE21}. 
Simple sets like polyhedrons and ellipsoids are often selected as uncertainty sets. 
The size of an uncertainty set can be determined by prior probability bounds \cite{DBE04, DBE18, DBE21, HON21, ZLI14}, 
by posterior probability bounds \cite{DBE21,ZLI14}, or by computed solutions \cite{ZLI12}.

\subsection*{Contribution}
This paper focuses on the individual chance constrained optimization problem with polynomial perturbation. 
We find an efficient robust approximation of CCO by replacing the chance constraint with a robust constraint on an ellipsoidal 
uncertainty set.
We propose semidefinite relaxation algorithms to solve the robust approximation globally. 
Analysis of convergent properties for these algorithms is carried out. 
In addition, a heuristic method is presented to find good uncertainty sets. 
Numerical experiments are given to show the efficiency of our approach.
The main contributions of this paper are summarized as follows.
\begin{itemize}
	\item
	We propose a robust approximation approach for individual chance constrained optimization with polynomial perturbation. 
	The approximation is constructed by replacing the ``hard'' chance constraint with a ``tractable'' robust constraint 
	over a convenient uncertainty set.
	Assume the chance constraint is linear in the decision variable.
	We show the approximation can be transformed into linear conic optimization with nonnegative polynomial cones. 
	The transformation is equivalent under some convex assumptions.
	
	\item
	We propose efficient semidefinite relaxation algorithms to solve these linear conic transformations.
	Under some general assumptions, these algorithms can converge to the global optimizers of the robust approximations.
	
	\item
	In addition, we introduce a heuristic method to find good uncertainty sets in robust approximations.
	We determine a preferred size of uncertainty set based on sampling information and 
	corresponding computed candidate solutions.
	Numerical experiments are presented to show the efficiency of our method.
\end{itemize}

The rest of this paper is organized as follows. 
Section \ref{sc:pre} reviews some basics for polynomial and moment optimization. 
Section \ref{sc:relaxalg} presents an algorithm for solving the robust approximation problem with a linear objective.  
Section \ref{sc:sosconvexpolynomial} studies the case where the objective function is an SOS-convex polynomial. 
Section \ref{sc:opti} discusses the construction of uncertainty sets. 
Section \ref{sc:numex} performs some numerical experiments and an application. Conclusions are summarized in Section \ref{sc:conclusion}.

\section{Preliminaries} \label{sc:pre}
\subsection*{Notation}
The symbol $\mathbb{R}$ (resp., $\mathbb{R}_{+}$, $\mathbb{N}$) denotes the set of 
real numbers (resp., nonnegative real numbers, nonnegative integers).  
The superscript $^T$ denotes the transpose of a vector or matrix.
The $e_{i} = (0,\ldots, 0,1,0,\ldots,0)^T$ denotes the $i$th unit vector in $\mathbb{R}^{n}$. 
For a vector $v\in\re^n$ and a scalar $R>0$, $B(v, R)$ denotes the closed ball centered at $v$ 
with radius $R$.
A symmetric matrix $W\in\re^{n\times n}$ is said to be positive semidefinite if 
$v^TWv\ge 0$ for all $v\in \re^n$, denoted as $W\succeq 0$. If $v^TWv>0$
for all nonzero vectors $v$, then $W$ is said to be positive definite, denoted as $W\succ 0$.  
Let $\xi := \left(\xi_1, \cdots, \xi_r\right)$. The $\mathbb{R}[\xi]$ stands for the ring of 
polynomials in $\xi$ with real coefficients and $\re[\xi]_d\subseteq \re[\xi]$ contains all 
polynomials of degrees at most $d$. For a polynomial $p$, $\operatorname{deg}(p)$ 
denotes its degree. For a tuple $p := \left(p_1, \cdots, p_m\right)$ of polynomials, 
$\operatorname{deg}(p)$ denotes the highest degree of all $p_i$, i.e., 
$\operatorname{deg}(p)=\max\{\operatorname{deg}(p_1),\cdots,\operatorname{deg}(p_m)\}$.
For a degree $d$, $[\xi]_d$ denotes the monomial vector with degrees up to $d$ and ordered
alphabetically, i.e.,
\begin{equation}\label{vectorofmonomials}
	[\xi]_d \,:=\, (1\quad\xi_1\quad\cdots\quad\xi_r\quad\xi_1^2\quad \xi_1\xi_2\quad \cdots\quad\xi_r^d)^T.
\end{equation}
Denote the power set 
$$\mathbb{N}_d^r :=\left\{\alpha:=(\alpha_1,\cdots,\alpha_r) \in \mathbb{N}^r:
\alpha_1+\cdots+\alpha_r\leq d\right\}.$$ 
For $t \in \mathbb{R}$, $\lceil t \rceil$ denotes the smallest integer that is no less than $t$. 
For a set $X\subseteq \re^n$, $\textbf{1}_{X}(\cdot)$ denotes the characteristic
function of $X$, i.e., 
$$\textbf{1}_{X}(x) \,:=\, \left\{\begin{array}{ll}
	1,& \mbox{if } x \in X, \\ 
	0,& \mbox{if } x \notin X.
\end{array}\right.$$

\subsection{SOS and nonnegative polynomials.}
A polynomial $\sigma \in \mathbb{R}[\xi]$ is said to be  
\textit{sum-of-squares}(SOS) if there are  some polynomials $s_1, \cdots,s_k \in \mathbb{R}[\xi]$ such that $\sigma := s_1^2 + \cdots + s_k^2$. We use $\Sigma[\xi]$ to denote the set of all SOS polynomials in $\mathbb{R}[\xi]$ and denote $\Sigma[\xi]_d:=\Sigma[\xi]\cap \re[\xi]_d$ for each degree $d$. 
For a polynomial $p \in \mathbb{R}[\xi]$, it is said to be {\it SOS-convex} if its Hessian matrix $\nabla^2 p$ is SOS, i.e., 
$\nabla^2 p=V(\xi)^TV(\xi)$ for a matrix-polynomial $V(\xi)$, 
and it is said to be {\it SOS-concave} if 
$-p$ is SOS-convex.

Let $g = (g_1,\ldots, g_m)$ be a tuple of polynomials. The quadratic module of $g$ is a polynomial cone defined as
\[
\qmod{g} \,\coloneqq\, \Sigma[\xi] + g_1 \cdot \Sigma[\xi]+ \cdots + g_m \cdot \Sigma[\xi].
\]
For an order $k$ such that $2k\ge \deg(g)$, the $k$th order truncated quadratic module of $g$ is given by
\[
\qmod{g}_{2k}\, \coloneqq\,  \Sigma[\xi]_{2k} + g_1\cdot\Sigma[\xi]_{2k-\deg(g_1)} + \cdots + 
g_m \cdot \Sigma[\xi]_{2k-\deg(g_m)}.
\]
For every $k$, we have the containment relation
\[ \qmod{g}_{2k}\subseteq\qmod{g}_{2k+2}\subseteq \qmod{g}.\]

Let $U\coloneqq\{\xi\in \mathbb{R}^{r}~\vert~ g(\xi)\geq 0\}$. Denote the nonnegative polynomial cone on $U$
\begin{equation}\label{eq:P(U)}
	\mathscr{P}(U) \,:=\, \{p\in\re[\xi] : p(\xi)\ge 0,\,\forall \xi\in U\}, 
\end{equation}
and its $d$th degree truncation $\mathscr{P}_d(U) = \mathscr{P}(U)\cap \re[\xi]_d$.
Clearly, $\qmod{g}$ is a subset of $\mathscr{P}(U)$. If there exists $R > 0$ such that 
$R^2 - \Vert \xi \Vert_2^2 \in \qmod{g}$, then $U$ must be compact and $\qmod{g}$ is said to be {\it archimedean}.
Under the archimedean assumption on $\qmod{g}$, if a polynomial $p\in\re[\xi]$ is positive on $U$, then
we have $p\in\qmod{g}$. The conclusion is called \textit{Putinar's Positivstellensatz} \cite{PUM93}. 

\subsection{Moment and localizing matrices}\label{subsc:mlm} 
Let $\mathbb{R}^{\mathbb{N}_{2k}^r}$ be the space of real vectors that are indexed by $\alpha \in \mathbb{N}_{2k}^r $. 
A real vector $y:=\left(y_{\alpha}\right)_{\alpha \in \mathbb{N}_{2k}^r }$ is called  \textit{truncated multi-sequence}(\textit{tms}) of degree $2k$.
It gives a  {\it Riesz functional} $\mathscr{L}_{y}$ acting on $\mathbb{R}[\xi]_{2k}$ as 
\[
\mathscr{L}_{y}\left(\sum\limits_{\alpha\in\mathbb{N}^{m}_{2k}}p_{\alpha}\xi^{\alpha}\right):=\, \sum\limits_{\alpha\in\mathbb{N}^{m}_{2k}}p_{\alpha}y_{\alpha}.
\]
For $p \in \mathbb{R}[\xi]_{2k}$ and $y \in \mathbb{R}^{\mathbb{N}_{2k}^r}$, we write that
\[
\left \langle p, y \right \rangle \,:= \, \mathscr{L}_{y}(p). 
\]
For a degree $d$, a tms $y\in\mathbb{R}^{\mathbb{N}_{d}^r}$ is said to admit a measure $\mu$ supported in a set $U$ if $y_{\alpha} = \int_U \xi^{\alpha} d\mu$ for every $\alpha\in\mathbb{N}_{d}^r$.
Denote 
\[
\mathscr{R}_d(U) \,:=\, \{y\in\re^{\mathbb{N}_d^r}: \mbox{$y$ admits a measure supported in $U$}\}.
\]
When $U$ is compact, $\mathscr{R}_d(U)$ is a closed convex cone that is dual to $\mathscr{P}_d(U)$, written as $\mathscr{R}_d(U) = \mathscr{P}_d(U)^*$.
Let $g \in \mathbb{R}[\xi]$ with $\deg(g)\le 2k$.
The $k$th order \textit{localizing matrix} of $g$ and $y$ is \
the symmetric matrix $L_{g}^{(k)}(y)$ such that
\[
\hbox{vec}(a_1)^T \left(L_{g}^{(k)}[y]\right)\hbox{vec}(a_2) \,=\, \mathscr{L}_{y}(ga_1a_2)
\]
for all $a_1, a_2\in \mathbb{R}[\xi]$ with degrees at most $k - \lceil \mbox{deg}(g)/2 \rceil$, where 
$\hbox{vec}(a_i)$ denotes the vector of  coefficients of $a_i$. When $g=1$ is the constant polynomial, 
$L_{1}^{(k)}[y]$ becomes the $k$th order \textit{moment matrix}
\[ M_{k}[y] \,:=\, L_{1}^{(k)}[y]. \]
For instance, when $r = 2, k=2$ and $g = 2-\xi_1^2-\xi_2^2$,
\[
L_g^{(2)}[y] = \begin{bmatrix}
	2y_{00}-y_{20}-y_{02} & 2y_{10}-y_{30}-y_{12} & 2y_{01}-y_{21}-y_{03}\\
	2y_{10}-y_{30}-y_{12} & 2y_{20}-y_{40}-y_{22} & 2y_{11} -y_{31}-y_{13}\\
	2y_{01}-y_{21}-y_{03} & 2y_{11}-y_{31}-y_{13} & 2y_{02}-y_{22}-y_{04}
\end{bmatrix}.
\]
Moment and localizing matrices can be used to represent dual cones of truncated
quadratic modules. 
For an order $k$, define the tms cone
\begin{equation}\label{eq:S(g)_2k}
	\mathscr{S}[g]_{2k} \,\coloneqq\, 
	\{ y\in\re^{\mathbb{N}_{2k}^r}: M_k[y]\succeq 0, L_g^{(k)}[y]\succeq 0\}.
\end{equation}
The $\mathscr{S}[g]_{2k}$ is closed convex for each $k$.
In particular, it is the dual cone of $\qmod{g}_{2k}$ \cite[Theorem~2.5.2]{Niebook}. That is,
\[
\mathscr{S}[g]_{2k} \,=\, (\qmod{g}_{2k})^* 
\,\coloneqq\, \left\{ y\in\re^{\mathbb{N}_{2k}^r}: \langle p, y\rangle \ge 0\,\forall p\in \qmod{g}_{2k} 
\right\}.
\]
In the above, the superscript $^*$ is the notation for dual cone.

\section{CCO with a linear objective } \label{sc:relaxalg}
Assume \reff{eq:cco} is an individual CCO problem with polynomial perturbation as in \reff{eq:poly_pert}.
Consider the relatively simple case that $f$ is a linear function, i.e., $f(x) = c^Tx$. The \reff{eq:cco} becomes 
\begin{equation}\label{eq:cco_linobj}
	\left\{\begin{array}{cl}
		\min\limits_{x\in X} &\,\, c^{T}x\\
		\st &\,\,  \mathbb{P}\{\xi: (Ax+b)^T[\xi]_d \ge 0\}\ge 1-\epsilon,
	\end{array}\right.
\end{equation}
where $A$ is a scalar matrix, $b,c$ are scalar vectors and $[\xi]_d$ is the monomial vector as 
in \reff{vectorofmonomials}. Let $\mu$ denote the mean value and $\Lambda$ denote the covariance matrix
of $\xi$.
For a proper size $\Gamma>0$, define the uncertainty set
\begin{equation}\label{eq:setU2}
	U \,=\, \{ \xi\in\re^r: \Gamma - (\xi-\mu)^T\Lambda^{-1}(\xi-\mu) \ge 0\}.
\end{equation}
We transform the chance constrain in \reff{eq:cco} into the robust constraint
\[
(Ax+b)^T[\xi]_d  \geq 0,~~\forall\xi\in U.
\]
It can also be interpreted as $(Ax+b)^{T}[\xi]_{d}$, as a polynomial in $\xi$, is nonnegative on $U$.
Recall that nonnegative polynomial cone $\mathscr{P}(U)$ as in \reff{eq:P(U)}. 
The robust approximation 
\reff{eq:robust_approximation} can be written as 
\begin{eqnarray}\label{primalproblem}
	\left\{ \begin{array}{rl}
		f^{\min}\,:=\, \min\limits_{x\in X} & c^{T}x  \\
		\st \, &  (Ax+b)^T[\xi]_d\in\mathscr{P}_d(U),
	\end{array}\right.
\end{eqnarray}
where $\mathscr{P}_d(U)$ is the $d$th degree truncation of $\mathscr{P}(U)$.
The \reff{primalproblem} is a linear conic optimization problem with a 
nonnegative polynomial cone. 
Since $\mathscr{P}_d(U)$ is hard to characterize in computations, 
we consider solving \reff{primalproblem} by semidefinite relaxations with quadratic modules.

\subsection{A semidefinite relaxation algorithm}
\label{A_Semidefinite_Relaxation_Algorithm}
Let $k_0:=\max\left\{\lceil d/2 \rceil,1\right\}$ and denote 
\[
g(\xi) \,:=\, \Gamma-(\xi-\mu)^T\Lambda^{-1}(\xi-\mu).
\]
Clearly, $\qmod{g}$ is archimedean since $U$ is always compact for given $\Gamma$. 
We can use truncated quadratic modules $\qmod{g}_{2k}$ to give good approximations of $\mathscr{P}(U)$.
Indeed, by \cite[Proposition~8.2.1]{Niebook}, we have
\[
\mbox{int}(\mathscr{P}_d(U))\,\subseteq \,\bigcup_{k=1}^{\infty} \qmod{g}_{2k}\cap \re[\xi]_d
\,\subseteq\, \mathscr{P}_d(U).
\]
The $k$th order SOS approximation of \reff{primalproblem} is
\begin{eqnarray} \label{sosrelax}
	\left\{\begin{array}{rl}  
		f_{k}^{sos} := \min\limits_{x\in \re^n} & c^{T}x  \\
		\st & (Ax+b)^{T}[\xi]_{d}  \in \qmod{g}_{2k}\cap \re[\xi]_d, \\
		& x\in X.
	\end{array}\right.
\end{eqnarray}
Recall the dual cone of $\qmod{g}_{2k}$ is the tms cone $\mathscr{S}[g]_{2k}$
as in \reff{eq:S(g)_2k}. Let $X^*$ denote the dual cone of $X$. 
We can get the dual problem of \reff{sosrelax}:
\begin{eqnarray}\label{momrealx}
	\left\{ \begin{array}{rl}
		f_{k}^{mom}:=\max\limits_{(y,z)} &   -b^Ty  \\
		\st &c-A^Ty\in X^{*},  \\
		& y=z{\mid}_{d} := (z_{\alpha})_{\alpha\in\mathbb{N}^{r}_{d}},\\
		& z \in \mathscr{S}[g]_{2k}. \\
	\end{array}\right.
\end{eqnarray}
Since $U$ is compact, the $\mathscr{R}_d(U)$ is closed. Then $\qmod{g}_{2k}\subseteq \mathscr{P}(U)$ implies 
\begin{equation}\label{eq:Rd}
	\mathscr{R}_d(U) \,:=\quad (\mathscr{P}_d(U))^* \,\subseteq\, (\qmod{g}_{2k}\cap \re[\xi]_d)^*.
\end{equation}
So we say \reff{momrealx} is the $k$th order relaxation of the dual problem of \reff{primalproblem}. 
Under the strong duality, the \reff{primalproblem} can be solved by  \reff{sosrelax} 
if the relaxation \reff{momrealx} is tight.
Assume $X, X^*$ are convex with semidefinite representations.
We summarize Algorithm~\ref{sdprelaxalgorithm} for solving \reff{primalproblem}.
\begin{algorithm}\label{sdprelaxalgorithm}
	For the robust approximation \reff{primalproblem}, 
	let $k = k_0$. Then do the following:
	\begin{itemize} [leftmargin=5em]
		\item[\textbf{Step 1:}]
		Solve \reff{sosrelax}--\reff{momrealx} for optimizers $x^*$ and $(y^*,z^*)$ respectively. 
		
		\item[\textbf{Step 2:}]
		If \reff{sosrelax}--\reff{momrealx} have no duality gap and there exists $t\in[k_0,k]$ such that
		\begin{equation}\label{rankcondition}
			\mbox{rank~} M_{t}[z^{*}] = \mbox{rank~} M_{t - 1}[z^{*}],
		\end{equation}
		then stop and output the optimizer $x^*$ and the optimal value $f^* = c^Tx^*$.
		Otherwise, let $k = k + 1$ and go to Step 1.
	\end{itemize}
\end{algorithm}

When $X, X^*$ have semidefinite representations, the primal-dual pair \reff{sosrelax}--\reff{momrealx} are semidefinite programs.
They can be solved globally with efficient computational methods like interior point methods.
The rank condition in \reff{rankcondition} is called {\it flat truncation} \cite{NIE13}.
It is a sufficient condition for the tms $y^* = z^*\vert_d \in\mathscr{R}_d(U)$,
which implies the relaxation \reff{momrealx} is tight.
The result is proved in the following theorem.
For convenience, let $f^{sos}_{k}, f^{mom}_{k}$ denote the optimal values of \reff{sosrelax}--\reff{momrealx} 
respectively at the order $k$.
\begin{theorem}\label{thm:tightrel}
	Assume $x^{*}$ and $(y^{*},z^{*})$ are optimizers of \reff{sosrelax}-- \reff{momrealx} respectively at some order $k$. 
	If $y^{*} \in\mathscr{R}_{d}(U)$ and there is no duality gap between \reff{sosrelax}--\reff{momrealx}, 
	i.e., $f_{k}^{sos}=f_{k}^{mom}$, 
	then $x^{*}$  is also a minimizer for (\ref{primalproblem}).
\end{theorem}
\begin{proof}
	Since $U$ in \reff{eq:setU2} is compact, $\mathscr{R}_d(U) = \mathscr{P}_d(U)^{*}$.
	The dual problem of \reff{primalproblem} is
	\begin{equation}\label{dualproblem}
		(D):\quad \left\{\begin{array}{cl}
			\max\limits_{y} & -b^Ty\\
			\st & c-A^Ty\in X^*,\\
			& y\in \mathscr{R}_d(U).
		\end{array}
		\right.
	\end{equation}
	Let $f^{\max}$ denote the optimal value of $(D)$.
	Since \reff{momrealx} is a relaxation of $(D)$, we have $f^{\max}\le f_k^{\max}$ for each $k$.
	If the optimizer $y^{*}\in\mathscr{R}_{d}(U)$ at some order $k$, then $y^{*}$ is also feasible for $(D)$, thus 
	\[ f_{k}^{mom} \,=\, -b^Ty^{*} \,\leq\, f^{\max}. \] 
	It implies that $f_k^{mom} = f^{\max}$ and that $y^*$ is the optimizer of $(D)$.
	Let $f^{\min}$ denote the optimal value of \reff{primalproblem}.
	By the weak duality between \reff{primalproblem} and $(D)$, we further have
	\[ f_{k}^{mom} \,=\, f^{\max} \,\leq\, f^{\min} \,\leq\,  f_{k}^{sos}. \]
	Suppose $f_{k}^{sos}=f_{k}^{mom}$. The above inequality implies
	\[
	-b^Ty^* \,=\, f^{\min} \,=\,  f_{k}^{sos} \,=\,  c^Tx^*.
	\]
	The $x^*$ is feasible for \reff{primalproblem}. 
	So it is also a minimizer of \reff{primalproblem}. 
\end{proof}
The above theorem guarantees Algorithm~\ref{sdprelaxalgorithm} returns the correct optimizer and optimal 
value of \reff{primalproblem} if it terminates in finite loops.
\begin{cor}
	If Algorithm~\ref{sdprelaxalgorithm} terminates in finite loops,
	then the output $x^*$ and $f^*$ are the correct optimizer and the optimal value of \reff{primalproblem}.
\end{cor}
\begin{proof}
	Let $k$ be the order in the terminating loop. The termination criteria requires $f_k^{sos} = f_k^{mom}$ and 
	the rank condition \reff{rankcondition} holds. 
	By \cite{NIE13}, the \reff{rankcondition} implies that the truncated tms $z^*\vert_t$ admits a measure supported in $U$.
	Since $t\ge k_0$ and $y^* = z^*\vert_{k_0}$, we must have $y^*\in\mathscr{R}_d(U)$.
	Then the conclusion is implied by Theorem~\ref{thm:tightrel}.
\end{proof}

\subsection{Convergent properties}\label{suc:converge} 
We study convergent properties of Algorithm~\ref{sdprelaxalgorithm}.
Major results from \cite{NIE14, NIE15} are used in our discussions,
which include insightful analysis of linear optimization with moment and nonnegative polynomial cones.

We first discuss the asymptotic convergence of Algorithm~\ref{sdprelaxalgorithm}. 
Recall that $h(x,\xi) = (Ax+b)^T[\xi]_d$.
For convenience, we denote the coefficient vector
\begin{equation}\label{eq:hx} 
	h(x)\,\coloneqq\, Ax+b.
\end{equation}
The optimization problem \reff{primalproblem} is said to be strictly feasible if there exists 
$\hat{x}\in\mbox{int}(X)$ such that $h(\hat{x}, \xi) >0$ for all $\xi\in U$.
At the relaxation order $k$, let $(y^{(k)}, z^{(k)})$ denote the optimizer of \reff{momrealx} and
let $f_k^{sos}, f_k^{mom}$ denote the optimal values of \reff{sosrelax}--\reff{momrealx} respectively.
\begin{theorem}\label{asmptotic convergence}
	Assume $X$ is a convex set with a semidefinite representation.
	Suppose \reff{primalproblem} is strictly feasible and its dual problem is feasible. 
	Then for all $k$ sufficiently large, $(y^{(k)},z^{(k)})$ exists and that 
	\[
	f^{sos}_{k} \,=\, f^{mom}_k \,\to\,  f^{\min}\quad \mbox{as}\quad k\to \infty.
	\]
\end{theorem}
\begin{proof}
	Suppose $\hat{x}$ is a strictly feasible point of \reff{primalproblem}.
	Then there exists $\epsilon_0>0$ such that $h(\hat{x},\xi)>\epsilon_0$ on $U$.
	Since $U$ is compact, there exists $\delta>0$ sufficiently small such that 
	$B(\hat{x}, \delta)\subseteq X$ and
	\[
	\forall x\in B(\hat{x},\delta):\quad
	\mbox{$h(x,\xi) > \epsilon_0$ on $U$ as a polynomial in $\xi$.}
	\]
	Since $\qmod{g}$ is archimedean, by \cite[Theorem~6]{NIE07}, 
	there exists $N_0 > 0$ such that $h(x,\xi)\in \qmod{g}_{2k}$ 
	for every $x\in B(\hat{x},\delta)$ and for all $k\ge N_0$.
	It implies that $\hat{x}$ is also a strictly feasible point of \reff{sosrelax} for all $k \geq N_{0}$. 
	Given the dual problem of \reff{primalproblem} is feasible, its relaxation \reff{momrealx} is also feasible.
	Hence for all $k$ big enough, there is a strong duality between \reff{sosrelax}--\reff{momrealx} and
	\reff{momrealx} is solvable with an optimizer $(y^{(k)}, z^{(k)})$.
	
	For an arbitrarily given $\epsilon_1\in(0,1)$, there exists a feasible point $\bar{x}$ of \reff{primalproblem}
	such that $f^{\min} \leq c^T\bar{x} < f^{\min} + \epsilon_1$.
	Denote 
	\[
	\tilde{x} \,:=\, (1-\epsilon_1)\bar{x}+ \epsilon_1 \hat{x}.
	\]   
	It is strictly feasible for \reff{primalproblem} (see as in \cite[Proposition~1.3.1]{DBE09}). 
	By previous arguments, we also have $\tilde{x}$ being feasible for \reff{sosrelax} 
	when $k$ is large enough. In this case,
	\[\begin{array}{c}
		f_k^{sos} \,\le\, c^{T}\tilde{x}\, = \, (1-\epsilon_1)c^T\bar{x}+\epsilon_1c^T\hat{x}\\ [2pt]
		\qquad \, <\, (1-\epsilon_1)(f^{\min}+\epsilon_1) + \epsilon_1 c^{T}\hat{x}.
	\end{array}\]
	Note $f^{\min} \leq f^{sos}_{k+1}\le f_k^{sos}$ for all $k$. 
	Since $\epsilon_1$ can be arbitrarily small, we can conclude that $f_k^{sos}\to f^{\min}$ as $k\to \infty$.
\end{proof}
The above conclusion still holds if $X$ is a convex set determined by SOS-concave polynomials. 
We refer to Section~\ref{sc:sosconvexpolynomial} for more details.
In addition, if we make some further assumptions on the optimizer(s) of \reff{primalproblem}, the finite convergence of Algorithm~\ref{sdprelaxalgorithm}
can be reached.

\begin{theorem}\label{finite convergence}
	Under all the assumptions of Theorem~\ref{asmptotic convergence},
	suppose \reff{primalproblem} is solvable with an optimizer $x^{*}$ such that $h(x^{*},\cdot)\in\qmod{g}_{2k}$
	for all $k>k_1$, and
	\begin{equation}\label{eq:assumptionproblem}
		\left\{\begin{array}{cl}
			\min\limits_{\xi\in\re^{r}} & h(x^*,\xi) = (Ax^*+b)^T[\xi]_d\\
			\st & g(\xi) \geq 0
		\end{array}\right.
	\end{equation}
	has only finitely many critical points $\zeta$ satisfying $h(x^*,\zeta)=0$.
	Then Algorithm~\ref{sdprelaxalgorithm} will terminate in finite loops.
\end{theorem}
\begin{proof}
	Under given conditions, the optimization problem \reff{primalproblem} and its dual \reff{dualproblem}
	are both solvable with optimizers $x^*$ and $y^*$ respectively. The strong duality holds that
	\[\begin{aligned}
		0 &\,=\, -b^Ty^*-c^{T}x^{*}\\
		& \,=\, -(Ax^{*}+b)^Ty^{*} -(c-A^{T}y^{*})^{T}x^{*}.
	\end{aligned}\]
	Since $c-A^Ty^*\in X^*$, $h(x^*,\cdot) \in \mathscr{P}_d(U)$ and $y^*\in\mathscr{R}_d(U)$, 
	we must have
	\[
	(Ax^*+b)^Ty^* \,=\, 0,\quad (c-A^{T}y^{*})^{T}x^{*} \,=\, 0.
	\]
	It implies that \reff{eq:assumptionproblem} has the optimal value $0$, which is achieved at every point 
	in the support of the measure admitted by $y^*$. The \reff{eq:assumptionproblem} is a polynomial optimization 
	problem. Its $k$th order Moment-SOS relaxation pair is
	\begin{eqnarray}\label{lasserre sos relaxation}
		\gamma_{k} \,:=\, \max\,\gamma \quad \st\quad h(x^*,\xi) - \gamma \in \qmod{g}_{2k},	
	\end{eqnarray}
	\begin{eqnarray}\label{lasserre moment relaxation}
		\gamma_{k}^{*} \,:=\, \min\, (Ax^*+b)^Tv \quad
		\st \quad v_{0} = 1,~~v \in \mathscr{S}[g]_{2k},
	\end{eqnarray}
	where $\mathscr{S}[g]_{2k}$ is given in \reff{eq:S(g)_2k}.
	Suppose $h(x^*,\cdot)\in\qmod{g}_{2k}$ for all $k\ge k_1$ and there are only finitely many critical points $\zeta$ of \reff{eq:assumptionproblem} such that $h(x^*,\zeta)=0$. 
	For all $k$ that is large enough, the relaxation pair \reff{lasserre sos relaxation}--\reff{lasserre moment relaxation}
	is tight, i.e., $\gamma_k = \gamma_k^* = 0$,  
	and every minimizer of \reff{lasserre moment relaxation} has a flat truncation \cite[Theorem~2.2]{NIE13}.
	
	Let $k\ge k_1$ be sufficiently large. The \reff{momrealx} is solvable with the optimizer $(y^{(k)}, z^{(k)})$. 
	If $(z^{(k)})_{0} >0$, then $z^{(k)}/(z^{(k)}_0)$ is feasible for \reff{lasserre moment relaxation} and it
	satisfies
	\[
	(Ax^*+b)^T\big( z^{(k)}/z^{(k)}_0\big) \,=\,0.
	\]
	This is because  $x^{*}$  is also a minimizer of \reff{sosrelax} and there is no duality gap between 
	\reff{sosrelax}--\reff{momrealx}.
	The above implies $z^{(k)}/(z^{(k)})_0$ is an optimizer \reff{lasserre moment relaxation}, 
	thus it has a flat truncation and satisfies the rank condition \reff{rankcondition}.
	If $(z^{(k)})_{0}=0$, then $\hbox{vec}(1)^T M_{k}[z^{(k)}]\hbox{vec}(1) = 0$,
	where $\hbox{vec}(1)$ denotes the coefficient vector of the scalar polynomial that is constantly one.
	Since $M_{k}[z^{(k)}]\succeq 0$, we must have $M_{k}[z^{(k)}]\hbox{vec}(1) = 0$. 
	By \cite[Lemma~5.7]{MLA09}, $M_{k}[z^{(k)}]\hbox{vec}(\xi^{\eta}) = 0$ for all $\vert \eta \vert \leq k-1$.
	Then for every $\alpha=\beta+\eta$ with $\vert\beta\vert, \vert \eta\vert \leq k-1$,
	\[
	(z^{(k)})_{\alpha} \,=\, \hbox{vec}(\xi^{\beta})^TM_{k}[z^{(k)}]\hbox{vec}(\xi^{\eta}) \,=\, 0.
	\] 
	It implies $z^{(k)}{\mid}_{2k-2}$ is a zero vector, which is naturally flat. 
	Therefore, the rank condition \reff{rankcondition} holds for $k$ sufficiently large.   
\end{proof}

\section{CCO with SOS-Convexity}\label{sc:sosconvexpolynomial}
In this section, we study a more general kind of individual CCO with polynomial perturbation.
In \reff{eq:cco}, assume $f(x)$ is an SOS-convex polynomial and the constraining set
\[
X \,=\, \{ x\in\re^n: u_1(x)\ge 0,\ldots, u_{m_1}(x)\ge 0\},
\]
where each $u_i$ is an SOS-concave polynomial. Let $U$ be a selected uncertainty set as in \reff{eq:setU2}.
We repeat the robust approximation \reff{eq:robust_approximation} as
\begin{eqnarray}\label{sos-convexproblem}
	\left\{ \begin{array}{cl}
		\min\limits_{x\in\re^{n}}& f(x)  \\
		\st &  (Ax+b)^{T}[\xi]_{d} \in \mathscr{P}_d(U),  \\
		& u_1(x) \geq 0, \cdots, u_{m_1}(x)\geq 0.
	\end{array}\right.
\end{eqnarray}
When $f, u_i$ are nonlinear, 
Algorithm~\ref{sdprelaxalgorithm}  cannot be directly applied to solve
\reff{sos-convexproblem}. 
Interestingly, we can always relax \reff{sos-convexproblem} into a linear conic
optimization problem with nonnegative polynomial cones and semidefinite constraints. Let
\[
d_1 \,:=\, \Big\lceil \frac{1}{2}\max \big\{\hbox{deg}(f),\hbox{deg}(u_1),\cdots,\hbox{deg}(u_{m_1})
\big\} \Big\rceil.
\]
We can relax \reff{sos-convexproblem} into the following problem:
\begin{eqnarray}\label{sos-convexrealx}
	\left\{\begin{array}{cl}
		\min\limits_{(x,w)} & \langle f, w\rangle \\
		\st &  (Ax+b)^{T}[\xi]_{d}  \in \mathscr{P}_d(U),  \\
		& \langle u_{i},w\rangle \geq 0 , i=1,\cdots,m_1,\\
		& M_{d_1}[w] \succeq 0,\quad w_0 = 1,  \\
		&  x=\pi(w),\,\,x\in \mathbb{R}^{n},\,\,
		w\in \mathbb{R}^{\mathbb{N}^{n}_{2d_1}},
	\end{array}\right.
\end{eqnarray}
where $\pi:\mathbb{R}^{\mathbb{N}_{2d_1}^{n}}\to \mathbb{R}^{n}$ is the projection map defined by
\[
\pi(w) \,:=\, (w_{e_1}, \cdots, w_{e_n}), \quad w \in \mathbb{R}^{\mathbb{N}_{2d_1}^{n}}.
\]
The relaxation is tight under SOS-convex assumptions.
\begin{theorem}\label{sos-convexconvergence}
	Suppose $f, -u_1,\ldots, -u_{m_1}$ are all SOS-convex polynomials. 
	The optimization problems \reff{sos-convexproblem}--\reff{sos-convexrealx} are equivalent in the sense
	that they have the same optimal value and a feasible point $(x^{*},w^{*})$ is a minimizer of \reff{sos-convexrealx} 
	if and only if $x^* = \pi(w^*)$ is a minimizer of \reff{sos-convexproblem}.
\end{theorem}
\begin{proof}
	Let $f_0, f_1$ denote the optimal values of \reff{sos-convexproblem}--\reff{sos-convexrealx} respectively.
	The $f_1\le f_0$ since \reff{sos-convexrealx} is a relaxation of \reff{sos-convexproblem}.
	Suppose $(\hat{x},\hat{w})$ is feasible for \reff{sos-convexrealx}.
	Since each $-u_i$ is SOS-convex, by Jensen's inequality \cite{LJB09}, we have
	\[
	-u_i(\hat{x}) \,=\, -u_i(\pi(\hat{w})) \,\le \, \langle -u_i, \hat{w}\rangle\,\le\, 0. 
	\]
	So $\hat{x} = \pi(\hat{w})$ must also be feasible for \reff{sos-convexproblem}.
	Also, since $f$ is SOS-convex, we have
	\[
	f_0 \,\le\, f(\hat{x}) \,=\, f(\pi(\hat{w})) \,\le\, \langle f,\hat{w}\rangle. 
	\]
	Then $f_0\le f_1$ since the above inequality holds for every feasible point of \reff{sos-convexrealx}. 
	So $f_0 = f_1$ and $x^* = \pi(w^*)$ is a minimizer of \reff{sos-convexproblem} if $(x^*,w^*)$ is an
	optimizer of \reff{sos-convexrealx}. Conversely, if $x^*$ is a minimizer of \reff{sos-convexproblem},
	then $w^* = [x^*]_{2d_1}$ is feasible for \reff{sos-convexrealx}, 
	and $(x^*,w^*)$ is an optimizer of \reff{sos-convexrealx}.
\end{proof}
As in Subsection \ref{A_Semidefinite_Relaxation_Algorithm}, 
we construct the $k$th order  SOS approximation of \reff{sos-convexrealx} with the truncated quadratic module of $g$  as follows
\begin{eqnarray}\label{Restricting_sos-convexrealx}
	\left\{\begin{array}{cl}
		\min\limits_{(x,w)}& \langle f,w\rangle \\
		\st &  (Ax+b)^{T}[\xi]_{d}  \in \qmod{g}_{2k},  \\
		& \langle u_{i},w\rangle \geq 0 ,\, i=1,\cdots,m_1,\\
		& M_{d_1}[w] \succeq 0,\quad w_{0}=1,   \\
		&  x=\pi(w),\quad x\in\mathbb{R}^{n},~~w\in \mathbb{R}^{\mathbb{N}^{n}_{2d_1}}.
	\end{array}\right.
\end{eqnarray}
Then we consider the dual problem of \reff{sos-convexrealx}.
Let $y\in\mathscr{R}_{d}(U)$, $\tau\in\mathbb{R}$,
$\lambda = (\lambda_1,\cdots,\lambda_{m_1})\in\re_+^{m_1}$ and $Q\succeq 0$.
The Lagrange function of (\ref{sos-convexrealx}) is 
\begin{align*}
	\mathcal{L}(w;y,\tau,\lambda,Q) \,=\, & 
	\langle f,w\rangle  - \langle A\pi(w)+b,y\rangle-\tau(w_{0}-1)\\
	&-\sum_{i=1}^{m_1}\lambda_i\langle u_i,w\rangle - \langle M_{d_1}[w],Q \rangle
\end{align*}
Denote $q(x) = f(x)-y^{T}Ax-\lambda^{T} u(x)-\tau$. One can simplify
\[
\mathcal{L}(w;y,\tau,\lambda,Q) \,=\, \langle q,w\rangle -
\langle M_{d_1}[w],Q\rangle+\tau-\langle  b,y\rangle.
\]
For $\mathcal{L}$ to be bounded from below for all tms $w$, we must have $\langle q,w\rangle =
\langle M_{d_1}[w],Q\rangle$. 
Note that $M_{d_1}[w]$ can be decomposed as
\[
M_{d_1}[w] \,=\, \sum\limits_{\alpha\in\mathbb{N}^{n}_{2d_1}}w_{\alpha}C_{\alpha},
\]
where each $C_{\alpha}$ is a symmetric matrix such that 
$[x]_{d_1}[x]_{d_1}^{T} := \sum_{\alpha\in\mathbb{N}^{n}_{2d_1}}x^{\alpha}C_{\alpha}$.
Then $\langle q,w\rangle = \langle M_{d_1}[w],Q\rangle$ can be reformulated as
\[
q(x) \,=\, \Big\langle \sum\limits_{\alpha\in\mathbb{N}^{n}_{2d_1}}x^{\alpha}C_{\alpha},Q\Big\rangle
\,=\, [x]_{d_1}^{T}Q[x]_{d_1},\quad \forall x\in \mathbb{R}^{n}.
\]
Let $u = (u_1,\ldots, u_{m_1})$.
Since $Q\succeq 0$, the above implies that $q(x)$ is an SOS polynomial in $x$, i.e.,
\[
q(x)=f(x)-y^{T}Ax-\lambda^{T} u(x)-\tau\in\Sigma[x]_{2d_{1}}.
\]  
Then the dual problem of \reff{sos-convexrealx} is
\begin{eqnarray} \label{sos-convexdualproblem}
	\left\{\begin{array}{cl}	
		\max\limits_{(\tau,\lambda, y)} & \tau - \langle b,y\rangle  \\ 
		\st &  f(x)-y^{T}Ax-\lambda^{T} u(x)-\tau\in\Sigma[x]_{2d_{1}}, \\
		& y \in \mathscr{R}_{d}(U),\,\, \lambda \in \mathbb{R}^{m_1}_{+},\,\, \tau \in \mathbb{R}.
	\end{array}\right.
\end{eqnarray}
The dual problem of \reff{Restricting_sos-convexrealx} is the $k$th order relaxation 
of the above problem, written as
\begin{eqnarray} \label{Relaxing_sos-convexdualproblem}
	\left\{\begin{array}{cl}	
		\max\limits_{(\tau,\lambda, y,z)} & \tau - \langle b,y\rangle  \\ 
		\st &  f(x)-y^{T}Ax-\lambda^{T} u(x)-\tau\in\Sigma[x]_{2d_{1}}, \\
		& y=z{\mid}_d,\,\,z \in \mathscr{S}[g]_{2k},\,\,\tau \in \mathbb{R},
		\,\, \lambda \in \mathbb{R}^{m_1}_{+}.
	\end{array}\right.
\end{eqnarray}
Since $\Sigma[x]_{2d_{1}}$ is semidefinite representable, 
\reff{Restricting_sos-convexrealx} and its dual \reff{Relaxing_sos-convexdualproblem} can be solved effectively 
by some semidefinite programming techniques. 
Then we summarize an algorithm for solving the robust approximation \reff{sos-convexproblem}
defined with SOS-convex polynomial.
\begin{algorithm}
	\label{Sdp_Relax_algorithm_for_SOS-convex_CCO}
	For the robust approximation \reff{primalproblem}, 
	let $k = k_0$. Then do the following:
	\begin{itemize} [leftmargin=5em]
		\item[\textbf{Step 1:}]
		Solve \reff{Restricting_sos-convexrealx} for a minimizer $(x^{*},w^{*})$ and 
		solve \reff{Relaxing_sos-convexdualproblem} for a maximmizer $(\tau^{*},\lambda^{*}, y^{*},z^{*})$.
		\item[\textbf{Step 2:}]
		If \reff{Restricting_sos-convexrealx} and  \reff{Relaxing_sos-convexdualproblem} have no duality gap
		and there exists an integer $t\in[k_0,k]$ such that $(\ref{rankcondition})$ holds, 
		then stop and output the optimizer $x^{*}=\pi(w^{*})$ and the optimal value $f(x^*)$. 
		Otherwise, let $k = k + 1$ and go to Step 2.
	\end{itemize}
\end{algorithm}

By Theorem~\ref{sos-convexconvergence}, the output optimizer $x^{*}$ in Step~2 
is also a minimizer of \reff{sos-convexproblem}. Since \reff{sos-convexproblem} is equivalent to
the linear conic optimization problem \reff{sos-convexrealx}, the convergent properties of Algorithm~
\ref{Sdp_Relax_algorithm_for_SOS-convex_CCO} can be proved with similar arguments as in
Subsection~\ref{suc:converge}.

\section{Size of uncertainty set}\label{sc:opti} 
In this section, we discuss how to choose the size of the uncertainty set $U$.
For convenience, denote the polynomial
\begin{equation}\label{Randomvecorofquantile}
	\Gamma(\xi) \,\coloneqq\, (\xi-\mu)^T\Lambda^{-1}(\xi-\mu). 
\end{equation}
Then we have $U = \{\xi\in\mathbb{R}^{r}\mid \Gamma(\xi) \leq \Gamma\}$ and 
\reff{eq:robust_approximation} can be expressed as 
\[
(\mbox{P}_{\Gamma}):\, \left\{\begin{array}{cl}
	\min\limits_{x\in X} & f(x)\\
	\st & h(x,\xi) \ge 0,\quad \mbox{if $\Gamma(\xi)\le \Gamma$}.
\end{array}\right.
\]
Clearly, $(\mbox{P}_{\Gamma})$ has a smaller feasible set with bigger $\Gamma>0$.
We want to find a small $\Gamma$ such that the optimizer of $(\mbox{P}_{\Gamma})$ is 
feasible for the original CCO. 
Assume the optimizer of \reff{eq:cco} is active at the chance constraint.
We introduce a heuristic method to compute the preferred set size.

Define the {\it $(1-\epsilon)$-quantile} of $\Gamma(\xi)$ as
\begin{equation}\label{eq:hatGamma}
	\hat{\Gamma}\,:=\,\inf\big\{ \Gamma: \mathbb{P}\{\xi: \Gamma(\xi)\le \hat{\Gamma}\} \,\ge\, 1-\epsilon\big\}. 
\end{equation}
When $\Gamma\ge \hat{\Gamma}$, every feasible point of $(\mbox{P}_{\Gamma})$ must also be feasible for \reff{eq:cco}.
So $\hat{\Gamma}$ gives an upper bound for the set size that we want.
However, it is very difficult to solve $\hat{\Gamma}$ analytically. 
In computations, it can be estimated using the quantile estimation method introduced by Hong et al in \cite{HON21}.
Suppose $\xi^{(1)},\cdots,\xi^{(N)}$ are independent and identically distributed (i.i.d.) samples of $\xi$. 
Up to proper reordering, we may have
\begin{equation}\label{orderstatistics}
	\Gamma(\xi^{(1)})\le \Gamma(\xi^{(2)})\le \cdots \le \Gamma(\xi^{(N)}).
\end{equation}
\begin{theorem}\cite[Lemma 3]{HON21}\label{quantileproof}
	Assume $\xi^{(1)},\cdots,\xi^{(N)}$ are i.i.d. samples of $\xi$ satisfying \reff{orderstatistics}.
	For given $\epsilon\in(0,1)$ and $\beta \in [0,1)$, the $\Gamma(\xi^{(L^*)})$ gives an upper bound on the 
	$(1-\epsilon)$-quantile of $\Gamma(\xi)$ with at least $(1-\beta)$-confidence level,
	where 
	\begin{equation}\label{statisticiupperbound}
		L^{*}:=\min\Big\{L\in\mathbb{N}: \sum_{i=0}^{L-1}\binom{N}{i}(1-\epsilon)^{i}\epsilon^{N-i}\geq 1-\beta,\, L\le N\Big\}.
	\end{equation}
	If \reff{statisticiupperbound} is unsolvable, then none of $\Gamma(\xi^{(i)})$ is a valid confidence upper bound. 
\end{theorem}

Suppose $x^{*}(\Gamma)$ is the optimizer of $(\mbox{P}_{\Gamma})$. 
The probability $\mathbb{P}\{\xi: h(x^{*}(\Gamma),\xi) <0\}$ can be estimated by sample averages; i.e.,
\begin{equation} \label{estimateprob}
	p_{vio}(x^*(\Gamma)) \,:=\, \frac{1}{\hat{N}}\sum\limits_{i=1}^{\hat{N}}\textbf{1}_{(-\infty,0)}
	\big(h(x^{*}(\Gamma),\zeta^{(i)})\big),
\end{equation}
where $\textbf{1}_{(-\infty,0)}(\cdot)$ denotes the characteristic function of $(-\infty,0)$
and $\zeta^{(1)},\ldots, \zeta^{(\hat{N})}$ are i.i.d. samples of $\xi$.

Assume $\hat{N}$ is big enough. Let $\rho\in(0,\epsilon)$ be a selected small tolerance. 
Since the optimizer of \reff{eq:cco} is active at the chance-constraint,
we want to determine a small $\Gamma$ such that $\vert p_{vio}(x^*(\Gamma))-\epsilon \vert\le \rho$.
If $p_{vio}(x^*(\Gamma))<\epsilon-\rho$,
then $x^*(\Gamma)$ is feasible to \reff{eq:cco} but it is expected to find a better
robust approximation with a smaller $\Gamma$.
If $p_{vio}(x^*(\Gamma))>\epsilon+\rho$, then $x^*(\Gamma)$ is not feasible for \reff{eq:cco} and
we need to select a larger $\Gamma$ such that the optimizer of \reff{eq:robust_approximation}
is feasible to \reff{eq:cco}.
We can repeat this process for several times until a proper $\Gamma$ is found such that
$p_{vio}(x^*(\Gamma))$ is equal or close enough to $\epsilon$.
This idea was used in Li and Floudas \cite{ZLI14, YYU17}.
We summarize it as a heuristic algorithm as follows.
\begin{algorithm}\label{iterativegorithm}
	For the CCO problem \reff{eq:cco}, 
	choose small scalars $\beta>0$, $\rho>0$, select sample sizes $N, \hat{N} (\hat{N}\gg N)$ 
	and let $\Gamma_{\mc{L}} = 0$ and $l=1$. 
	Then do the following
	\begin{itemize}[leftmargin=5em]
		\item[\textbf{Step 1:}] 
		Generate $N+\hat{N}$ i.i.d. samples of $\xi$ following the distribution $\mathbb{P}$.
		Select $N$ of them to be labeled with $\xi^{(i)}$ and ordered by \reff{orderstatistics}.
		Denote the rest samples by $\{\zeta^{(i)}\}_{i=1}^{\hat{N}}$.

		\item[\textbf{Step 2:}] 
		Determine $L^*$ in \reff{statisticiupperbound} with $\beta$.
		If \reff{statisticiupperbound}  is unsolvable, update $\beta\coloneqq 2\beta$.
		Set $\Gamma_1\coloneqq \Gamma(\xi^{(L^*)})$ and 
		$\Gamma_{\mc{U}} \coloneqq \Gamma(\xi^{(L^*)})$.
		Then construct the uncertainty set $U$ as in \reff{eq:setU2} with set size $\Gamma = \Gamma_{l}$.
		
		\item[\textbf{Step 3:}]
		Solve \reff{eq:robust_approximation} for a minimizer $x^{*}(\Gamma)$ and
		estimate the probability of constraint violation $p_{vio}(x^*(\Gamma))$ as in \reff{estimateprob}.
		If $\vert p_{vio}(x^*(\Gamma)) - \epsilon \vert \leq \rho$, 
		then stop and output the set size  $\Gamma$.
		Otherwise, go to the next step.
		
		\item[\textbf{Step~4:}]
		If $p_{vio}(x^*(\Gamma)) < \epsilon$, update $\Gamma_{\mc{U}} \coloneqq \Gamma$.
		Otherwise, update $\Gamma_{\mc{L}} \coloneqq\Gamma$.
		Let $l\coloneqq l+1$, update $\Gamma_l \coloneqq \frac{1}{2}(\Gamma_{\mc{L}}+\Gamma_{\mc{U}})$ 
		and go back to Step 3.
	\end{itemize}
\end{algorithm}

\section{Numerical experiments}\label{sc:numex}
In this section, we represent numerical experiments to show the efficiency of our method.
The computations are implemented in {\tt MATLAB} R2019a, running on a desktop computer with 4.0GB RAM and Intel(R) Core(TM) i3-4160 CPU using software {\tt Gloptipoly~3} \cite{HEN09}, 
{\tt Yalmip} \cite{JLO05}, and $\mathtt{SeDuMi}$ \cite{STU99}.
All associated samples are generated by {\tt MATLAB} commands following the specified distribution. For each example, we set the confidence level $\beta=0.05$, 
the sample sizes  $N=100$ and $\hat{N} = 10^{6}$,
and the tolerance parameter $\rho = 10^{-6}$, unless stated otherwise. 
We use $\Gamma^{*}$ to denote the terminated uncertainty set size in Algorithm~\ref{iterativegorithm}. 
The $f^{*}$ and $x^{*}$ are used to represent the optimal value and optimizer of 
\reff{eq:robust_approximation} with the set size $\Gamma^*$.
For neatness, we only display four decimal digits for numerical results.

\subsection{Efficincy of Algorithm~\ref{iterativegorithm}}
We give an example to show the efficiency of Algorithm~\ref{iterativegorithm}.
\begin{example}\label{example_of_alg3}
	{\rm
		Consider the individual CCO problem
		\[
		\left\{\begin{array}{cl}
			\min\limits_{x\in \re^3}&  f(x)=x_1 + x_2 + x_3   \\
			\st & \mathbb{P}\{\xi : h(x,\xi) \geq 0 \} \geq 1-\epsilon,  \\
			& x_1 - 2x_2 + 2x_3 \geq 2,
		\end{array}\right.
		\]
		where $\xi$ follows multivariate Gaussian distribution with  mean and covariance 
		\[
		\mu=\begin{bmatrix} 1\\1\\ 2\end{bmatrix},\quad
		\Lambda=\begin{bmatrix}2 &1 &0.5\\1&2&0.4\\0.5&0.4&3\end{bmatrix}.
		\]
		and the random constraining function
		\[
		\begin{array}{c}
			h(x,\xi) \,=\, (3x_1+2x_2+2x_3)\xi_1^4+(x_1+2x_2+2x_3-3)\xi_2^2\xi_3^2+(x_1- 2x_2)\xi_1^2\xi_2\\
			\qquad +(x_2+3x_3)\xi_2+(3x_2+x_3)\xi_3+(2x_1+4x_2+x_3).
		\end{array}
		\]
		We make numerical experiments for different $\epsilon, \beta$ and $N$.
		For each pair $(\beta, N)$, we made $100$ sampling instances.
		
		\noindent
		(i).
		First, we explore the performance of the initial set size $\Gamma(\xi^{(L^*)})$ 
		implied by Theorem~\ref{quantileproof}. For each instance, 
		we generate an uncertainty set with the initial size bound.
		The violation probability
		$p_{vio}(U):=\mathbb{P}\{\xi:\xi\not\in U\}$ can be evaluated analytically with Gaussian distributions.
		When $p_{vio}(U)>\epsilon$, we say the instance is {\it unsuitable}. 
		Denote $\#$ as the number of unstable instances and
		\[
		\overline{p_{vio}(U)}: \mbox{average of $p_{vio}(U)$},\quad
		p_{vio}^{\sigma}(U): \mbox{standard deviation of $p_{vio}(U)$}
		\] 
		among these $100$ instances. 
		We report the computational results in Table \ref{tab_set_size_0.05}.
		\begin{table}[ht]
				\caption{Numerical results for the probability violation of the initial size for Example \ref{example_of_alg3}}
				\label{tab_set_size_0.05}
				\centering
				\begin{tabular}{c|cccc|cccc}
					\toprule & 
					\multicolumn{4}{c|}{$\epsilon=0.05$} & \multicolumn{4}{c}{$\epsilon = 0.01$}\\ [2pt]
					\midrule
					$\beta$ & $N$ & $\overline{p_{vio}(U)}$ & $p_{vio}^{\sigma}(U)$ &$\#$ & $N$ 
					&$\overline{p_{vio}(U)}$& $p_{vio}^{\sigma}(U)$ & $\#$ \\ 
					\midrule
					\multirow{5}{*}{$0.01$}  & $90$   & $0.0103$ & $0.0106$ & $0$  \\ [2pt]
					& $500$  & $0.0286$ & $0.0075$ & $1$   & $459$  & $0.0023$   & $0.0024$ & $2$   \\[2pt]
					& $1000$ & $0.0348$ & $0.0050$ & $0$ & $1000$ & $0.0031$   & $0.0018$ & $0$ \\ [2pt]
					& $5000$ & $0.0429$ & $0.0030$ & $3$  & $5000$ & $0.0069$   & $0.0011$ & $0$     \\[2pt]
					& $10000$& $0.0451$ & $0.0022$ & $3$   & $10000$& $0.0080$   & $0.0009$ & $1$       \\
					\midrule
					\multirow{5}{*}{$0.05$}  & $59$   & $0.0166$ & $0.0182$ & $6$   & $299$  & $0.0032$ & $0.0030$ & $4$  \\[2pt]
					& $500$  & $0.0338$ & $0.0091$ & $5$   & $500$  & $0.0042$ & $0.0028$ & $4$    \\[2pt]
					& $1000$ & $0.0387$ & $0.0052$ & $3$  & $1000$ & $0.0053$ & $0.0024$ & $5$    \\[2pt]
					& $5000$ & $0.0451$ & $0.0032$ & $8$   & $5000$ & $0.0081$ & $0.0013$ & $8$   \\[2pt]
					& $10000$& $0.0464$ & $0.0019$ & $3$  & $10000$ & $0.0086$ & $0.0009$ & $8$    \\
				    \bottomrule
				\end{tabular}
			\end{table}
			It is easy to observe that for both risk levels, 
			$p_{vio}(U)$ goes approaching $\epsilon$ as $N$ increases, 
			while the number of unsuitable cases remains small.
			Indeed, even if the sample size $N$ is small, 
			the obtained uncertainty set is often suitable in our computational results.
			
			\noindent
			(ii).
			Then, we apply Algorithm~\ref{iterativegorithm} to determine a better set size than $\Gamma(\xi^{(L^*)})$ in (i).
			For each experiment, $p_{vio}(x^*(\Gamma))$ is estimated with $\hat{N}=10^{6}$ samples of $\xi$. 
			Denote the optimal value of \reff{eq:robust_approximation} by $f_{I}$ and $f_{T}$ respectively 
			for the initial set size and the terminating set size. We write
			\[
			\overline{f_I}: \mbox{average of $f_I$},\quad
			f_I^{\sigma}: \mbox{standard deviation of $f_I$}
			\]
			and $\overline{f_T},f_T^{\sigma}$ similarly.
			The computational results are reported in Table~\ref{tab_optimal_value_0.05}.
			\begin{table}[ht]
				\begin{center}
					\caption{Numerical results for Example $\ref{example_of_alg3}$}\label{tab_optimal_value_0.05} 
				\begin{tabular}{c|c|cccccc}
					\toprule
					$\epsilon$ & $\beta$ & $N$ & $\overline{f_I}$ & $f_I^{\sigma}$ & $\overline{f_T}$ & $f_T^{\sigma}$ \\
					\midrule
					\multirow{12}{*}{$0.05$} & \multirow{5}{*}{$0.01$}  & 90   & 2.5147 & 0.0639 & 1.2844 & $4.8869 \cdot 10^{-4}$  \\[2pt]
					&   & 500  & 2.4227 & 0.0211 & 1.2843 & $5.1423 \cdot 10^{-4}$   \\[2pt]
					&   & 1000 & 2.4055 & 0.0122 & 1.2845 & $4.7491 \cdot 10^{-4}$    \\[2pt]
					&   & 5000 & 2.3862 & 0.0066 & 1.2843 & $5.0483 \cdot 10^{-4}$     \\[2pt]
					&   & 10000& 2.3813 & 0.0048 & 1.2844 & $4.8448 \cdot 10^{-4}$      \\
					\cline{2-7}
					& \multirow{5}{*}{$0.01$}  & 59   & 2.4871 & 0.0765 & 1.2843 & $4.7734 \cdot 10^{-4}$   \\[2pt]
					&    & 500  & 2.4090 & 0.0224 & 1.2843 & $5.0555 \cdot 10^{-4}$    \\[2pt]
					&    & 1000 & 2.3961 & 0.0120 & 1.2843 & $5.8172 \cdot 10^{-4}$     \\[2pt]
					&    & 5000 & 2.3814 & 0.0069 & 1.2843 & $5.2490 \cdot 10^{-4}$      \\[2pt]
					&    & 10000& 2.3785 & 0.0041 & 1.2844 & $4.7893 \cdot 10^{-4}$       \\
					\midrule
					\multirow{10}{*}{$0.01$} & \multirow{5}{*}{$0.01$}  & 459  & 2.5918 & 0.0567 & 1.3458 & $6.2093 \cdot 10^{-5}$   \\ [2pt]
					&     & 1000 & 2.5593 & 0.0292 & 1.3458 & $5.0232 \cdot 10^{-5}$    \\ [2pt]
					&     & 5000 & 2.5123 & 0.0087 & 1.3458 & $4.5771 \cdot 10^{-5}$     \\ [2pt]
					&     & 10000& 2.5046 & 0.0063 & 1.3458 & $6.0832 \cdot 10^{-5}$      \\ [2pt]
					\cline{2-7}
					& \multirow{5}{*}{$0.05$}  & 299  & 2.5696  & 0.0485 & 1.3458 & $5.4640 \cdot 10^{-5}$     \\ [2pt]
					&      & 500  & 2.5472  & 0.0350 & 1.3458 & $5.3941 \cdot 10^{-5}$      \\[2pt]
					&      & 1000 & 2.5310  & 0.0240 & 1.3458 & $5.7009 \cdot 10^{-5}$       \\[2pt]
					&      & 5000 & 2.5037  & 0.0083 & 1.3458 & $4.6144 \cdot 10^{-5}$         \\ [2pt]            
					&      & 10000& 2.5004  & 0.0059 & 1.3458 & $5.4902 \cdot 10^{-5}$      \\
					\bottomrule
				\end{tabular}
			\end{center}
		\end{table}
		By applying Algorithm~\ref{iterativegorithm}, the approximation quality is improved 
		as the terminated optimal value average $\overline{f_T}$ is near half of the initial value
		$\overline{f_I}$. 
		On the other hand, the difference among $f_T^\sigma$ for different $(\beta,N)$ is tiny,
		which encourages us to use relatively smaller sample sizes to construct uncertainty sets.
	}
\end{example}

\subsection{Robust approximation of CCO}
We give some examples of individual CCO problems with linear objective functions.
For a given uncertainty set, we apply Algorithm~ \ref{sdprelaxalgorithm} 
to solve the robust approximation \reff{eq:robust_approximation}.
\begin{example}{\rm
		Consider the individual CCO problem
		\begin{eqnarray*}
			\left\{\begin{array}{cl}
				\min\limits_{x\in\mathbb{R}^{3}} & f(x)=2x_1+3x_2+x_3  \\
				\st & \mathbb{P}\{\xi: \, h(x,\xi)\geq 0\} \geq 0.75,\\
				& 4-x_1-x_2-x_3 \geq 0, \\
				& 2-x_1+2x_2-x_3 \geq 0,
			\end{array}
			\right.
		\end{eqnarray*}
		where the random constraining function
		\begin{eqnarray*}
			& &h(x,\xi)= (-3x_1+2x_2)\xi_1^4 + (x_1+3x_3+1)\xi_2^4 + (-3x_2+2x_3+3)\xi_1^2\xi_2 \\ 
			& & \qquad \qquad+(x_1+2x_3)\xi_2^2\xi_3+(2x_1+x_2-2x_3),
		\end{eqnarray*}
		and $\xi_1$, $\xi_2$ and $\xi_3$ are independent and identically distributed random variables with the uniform distribution on $[0,2]$. 
		The mean vector $\mu$ and covariance matrix $\Lambda$ of $\xi$ are 
		\[
		\mu \,=\, \mathbf{1}_{3\times 1},\quad \Lambda\,=\, \frac{1}{3}I_3,
		\]
		where $\mathbf{1}$ denotes the matrix of all ones and $I_n$ denotes the $n$-dimensional identity matrix.
		By Theorem~\ref{quantileproof}, we computed the initial set size and the probability violation
		\[
		\Gamma_1 \,=\, 4.4388,\quad
		p_{vio}(x^*(\Gamma_1)) \,=\, 0.0012.
		\]
		It took $24.1242$ seconds for Algorithm~\ref{iterativegorithm} to terminate after $21$ loops,
		where Algorithm \ref{sdprelaxalgorithm} stops at the initial loop in each inner iteration.
		In the terminating loop, we get 
		\[\begin{array}{c}
			\Gamma^* \,=\, 1.5387,\quad p_{vio}(x^*) \,=\, 0.2500,\\
			f^{*} \,= \, -1.6382,\quad x^* \,=\, (0.0531,\, -0.4513,\, -0.1781)^T.
		\end{array}\]
	}
\end{example}

\begin{example}{\rm
		Consider the individual CCO problem
		\begin{eqnarray*}
			\left\{\begin{array}{cl}
				\min\limits_{x\in\mathbb{R}^{4}}& f(x)=x_1+2x_2+3x_3+x_4      \\ 
				\st & \mathbb{P}\{\xi:\, h(x,\xi)  \geq 0\} \geq 0.80, \\
				& \begin{bmatrix} 8+3x_2-4x_4 & 5+x_3  & 2x_2-3x_4 \\
					5+x_3 & 10+2x_2 & -x_1-3x_2+3x_3 \\
					2x_2-3x_4 & -x_1-3x_2+3x_3 & 3+3x_1+8x_4    
				\end{bmatrix} \succeq 0,
			\end{array}\right.
		\end{eqnarray*}  
		where the constraining random function
		\begin{eqnarray*}
			& &h(x,\xi) =(x_1 + x_3+3x_4) \xi_1^5+(-x_2+3 x_3+2x_4)\xi_2^5+(x_1+3x_2+2x_4+2)  \\
			& & \qquad \qquad \xi_3^2\xi_4\xi_5^2+(2x_1+2x_2+x_4-5)\xi_3\xi_4\xi_5+(x_1+x_2+2x_3-x_4), 
		\end{eqnarray*}
		and each $\xi_{i}$ independently follows $t$-distribution with degrees of freedom $\bar{\nu}=3$. 
		The mean vector $\mu$ and covariance matrix $\Lambda$ of $\xi$ are ($\mathbf{0}$ is the zero matrix)
		\[
		\mu\,=\,  \mathbf{0}_{5\times 1},\quad \Lambda\,=\,3I_5.
		\]
		By Theorem~\ref{quantileproof}, we computed the initial set size and the probability violation
		\[
		\Gamma_1 \,=\, 9.0544,\quad
		p_{vio}(x^*(\Gamma_1)) \,=\, 0.0155.
		\]
		It took $40.8692$ seconds for Algorithm~\ref{iterativegorithm} to terminate after $23$ loops,
		where Algorithm \ref{sdprelaxalgorithm} stops at the initial loop in each inner iteration.
		In the terminating loop, we get 
		\[\begin{array}{c}
			\Gamma^* \,=\, 1.2693,\quad p_{vio}(x^*) \,=\, 0.2000,\\
			f^{*} \,= \, 0.7784,\quad x^* \,=\, (5.1776,\, -2.0061,\, 0.4772,\, -1.8185)^T.
		\end{array}\]
	}
\end{example}

\begin{example}{\rm
		Consider the individual CCO problem 
		\begin{eqnarray*}
			\left\{\begin{array}{cl}
				\min\limits_{x\in\mathbb{R}^{3}}& f(x) = -2x_1-3x_2+x_3       \\ 
				\st & \mathbb{P}\{\xi:~h(x,\xi)  \geq 0\} \geq 0.90, \\
				& 2-x_1+2x_2-x_3 \geq 0,  \\
				& 1-x_1^2-x_2^2-x_3^2 \geq 0,
			\end{array}
			\right.
		\end{eqnarray*}  
		where the random constraining function
		\begin{eqnarray*}
			& &h(x,\xi)=(3x_1-6x_2)\xi_1^4\xi_2^2+(x_1-x_3)\xi_1^3\xi_2^2+(x_1+3)\xi_1^2  \\
			& &\qquad \qquad +(x_3+2)\xi_2^2 +3x_2\xi_1-4x_3\xi_2,
		\end{eqnarray*}
		and $\xi_1$ and $\xi_2$ are two independent random variables that  follow the 
		exponential distribution with parameter $\bar{\lambda}=1$ and $\bar{\lambda}=2$, respectively.
		The mean vector $\mu$ and the covariance matrix $\Lambda$ of $\xi$ are
		\[ 
		\mu\,=\,\begin{bmatrix} 1\\2 \end{bmatrix},\quad 
		\Lambda \,=\, \begin{bmatrix} 1&0\\0&4\end{bmatrix}.
		\]
		By Theorem~\ref{quantileproof}, we computed the initial set size and the probability violation
		\[
		\Gamma_1 \,=\, 5.3688,\quad
		p_{vio}(x^*(\Gamma_1)) \,=\, 0.0000.
		\]
		It took $18.0012$ seconds for Algorithm~\ref{iterativegorithm} to terminate after $20$ loops,
		where Algorithm \ref{sdprelaxalgorithm} stops at the initial loop in each inner iteration.
		In the terminating loop, we get 
		\[\begin{array}{c}
			\Gamma^* \,=\, 1.2693,\quad  p_{vio}(x^*) \,=\, 0.1000,\\
			f^{*} \,= \, -3.5249,\quad x^* \,=\, (0.7656,\,0.5576,\,-0.3208)^T.
		\end{array}\]
	}
\end{example}

\begin{example}{\rm
		Consider the individual CCO problem  
		\begin{eqnarray*}
			\left\{\begin{array}{cl}
				\min\limits_{x\in\mathbb{R}^{2}} &f(x)= x_1+2x_2   \\
				\st & \mathbb{P} \{\xi : \, h(x,\xi) \geq 0\} \geq 0.95, \\
				& 3+2x_1-x_2 \geq 0,            \\
				& 1-x_1+x_2\geq 0,
			\end{array}\right.
		\end{eqnarray*}
		where the random constraining function
		\begin{eqnarray*}
			&  &h(x,\xi)=x_1\xi_1^4+3x_2\xi_2^4+2x_1\xi_1\xi_2+(3x_1-3x_2)\xi_2^2  \\
			& & \qquad\qquad+(x_2+3)\xi_1+(-x_1+x_2-2)\xi_2+(3x_1+4x_2),
		\end{eqnarray*}
		and the random vector $\xi$ follows  an empirical probability measure supported 
		by $1000$ samples.
		Those samples were randomly generated from the standard Gaussian distribution.
		The mean $\mu$ and  covariance matrix $\Lambda$ of $\xi$ are 
		\[
		\mu \,=\, \begin{bmatrix}0.0676\\0.0132
		\end{bmatrix},\quad \Lambda\, =\,\left[\begin{array}{rr}
			0.9887& -0.0057\\ -0.0057 & 0.9848
		\end{array}\right].
		\]
		By Theorem~\ref{quantileproof}, we computed the initial set size and the probability violation  with $\hat{N}=1000$ samples
		\[
		\Gamma_1 \,=\, 6.4948,\quad
		p_{vio}(x^*(\Gamma_1)) \,=\, 0.0000.
		\]
		It took $3.0562$ seconds for Algorithm~\ref{iterativegorithm} to terminate after $11$ loops,
		where Algorithm \ref{sdprelaxalgorithm} stops at the initial loop in each inner iteration.
		In the terminating loop, we get 
		\[\begin{array}{c}
			\Gamma^* \,=\, 0.7548,\quad  p_{vio}(x^*) \,=\, 0.0500,\\
			f^{*} \,= \, 1.0895,\quad x^* \,=\, (1.0298,\, 0.029)^T.
		\end{array}\]
		
		The $x^*$ is the optimizer of \reff{eq:robust_approximation} where the uncertainty set $U$ has the size $\Gamma^*$.
		Interestingly, $U$ with the size $\Gamma^*$ only contains $303$ samples of $\xi$; i.e.,
		\[ \mathbb{P}\{\xi:\xi\in U\} = 0.303\quad <\quad0.95=\mathbb{P}\{\xi: h(x^*,\xi)\ge 0\}. \]
		The probability difference implies that $\mathbb{P}\{\xi:\xi\in U\}\ge 1-\epsilon$ is sufficient but not necessary  
		for the chance constraint in \reff{eq:cco} to hold.
		In other words, there exists $U$ such that $\mathbb{P}\{\xi \, :\,  \xi \in U \} < 1 - \epsilon$ 
		but $\mathbb{P}\{\xi \, :\,  h(x^{*},\xi)\geq 0 \} \geq 1 - \epsilon$ for the optimizer $x^{*}$ of (1.3). 
	}
\end{example}

Then we give some examples of individual CCO problems defined with SOS-convex polynomials.
For a given uncertainty set, we apply Algorithm~\ref{Sdp_Relax_algorithm_for_SOS-convex_CCO}
to solve the robust approximation \reff{eq:robust_approximation}.
\begin{example}{\rm
		Consider the individual CCO problem
		\begin{eqnarray*}
			\left\{\begin{array}{cl}
				\min\limits _{x\in\mathbb{R}^{5}} & f(x)=4x_1^4+6x_2^2+x_3+3x_4+x_5 \\
				\st & \mathbb{P}\{\xi: \, h(x,\xi)\geq 0\}\geq 0.90,   \\
				& u_{1}(x)=8-x_1^2-x_2^2-x_3^2-x_4^2-x_5^2 \geq 0,\\
				& u_{2}(x)=10-3x_1^4-6x_2^2-2x_3^4+6x_4-3x_5\geq 0,
			\end{array}\right.       
		\end{eqnarray*}
		where the random constraining function
		\begin{eqnarray*}
			\begin{array}{ll}
				h(x,\xi)\,=\, &(3x_1+2x_2+2x_4) \xi_1^4 +(x_2-2x_4+2x_5)            
				\xi_2^2\xi_3^2 +(x_3-2x_4)\xi_1^2\xi_4+(3x_2\\
				&-x_3-3x_5)\xi_3+(2x_2-3x_5)\xi_4+(2x_1+4x_2+x_3-5x_4-10x_5),
			\end{array}
		\end{eqnarray*}
		and the random vector $\xi$ follows $t$-distribution with freedom degree, location and the scale matrix 
		$$\bar{\nu} \,=\, 4,\quad 
		\bar{\mu} \,=\, \begin{bmatrix}
			1\\1\\2\\3
		\end{bmatrix},\quad 
		\bar{\Lambda}\,=\, \begin{bmatrix}
			4&2&0&1\\2&3&0&1\\0&0&2&3\\1&1&3&6
		\end{bmatrix}.$$ 
		Then the mean and covariance of $\xi$ are 
		\[
		\mu \,=\, \bar{\mu},\quad \Lambda\,=\, \frac{\bar{\nu}}{\bar{\nu}-2}\bar{\Lambda}.
		\]
		It is easy to verify that $f$, $-u_1$ and $-u_2$ are all SOS-convex polynomials. 
		By Theorem~\ref{quantileproof}, we computed the initial set size with $N = 300$ samples and the probability violation
		\[
		\Gamma_1 \,=\, 48.6056,\quad
		p_{vio}(x^*(\Gamma_1)) \,=\, 0.0000.
		\]
		It took $21.2167$ seconds for Algorithm~\ref{iterativegorithm} to terminate after $21$ loops,
		where Algorithm~\ref{Sdp_Relax_algorithm_for_SOS-convex_CCO} stops at the initial loop in each inner iteration.
		In the terminating loop, we get 
		\[\begin{array}{c}
			\Gamma^* \,=\, 3.2416,\quad  p_{vio}(x^*) \,=\, 0.0100,\quad
			f^{*} \,= \, -3.7496,\\
			x^* \,=\, (0.5603,\,-0.0467,\,-1.3485,\,-0.7668,\,-0.5080)^T.
		\end{array}\]
	}
\end{example}

\begin{example}{\rm
		Consider the individual CCO problem  
		\begin{eqnarray*}
			\left\{\begin{array}{cl}
				\min\limits_{x\in\mathbb{R}^{4}}& f(x)=6x_1^4+x_3^2+3x_4^2+5x_2  \\ 
				\st & \mathbb{P}\{\xi:~h(x,\xi)  \geq 0\}\geq 0.85,  \\
				& u_1(x)=11-(x_1+x_2)^4-2x_3^4-(3x_3-x_4)^2+5x_2 \geq 0,      \\
				& u_2(x)=6 -(x_2-x_3)^2-3x_2x_3+4x_3+3x_4 \geq 0,
			\end{array}\right.
		\end{eqnarray*}  
		where the random constraining function
		\begin{eqnarray*}
			\begin{array}{ll}
				h(x,\xi) \,=\, &(8x_1+x_2+6x_3)\xi_1^2\xi_2^4+(x_3-2x_4)\xi_3^2\xi_4^2+(x_2+3x_4)\xi_2\xi_3^2\\
				&+(x_1-2x_2+x_3)\xi_1\xi_2+(2x_1+3x_3+1)\xi_3\xi_4+(8x_1-4x_2-2x_3),
			\end{array}
		\end{eqnarray*}
		and the random variables $\xi_i,i=1,2,3,4$ are independent to each other.
		Assume $\xi_1$ is a beta distribution with shape parameters $\bar{\alpha}=\bar{\beta}=2$, 
		$\xi_2$ is a gamma distribution with shape parameter $\bar{k}=2$ and scale parameter $\bar{\theta}=1$, 
		$\xi_3$ is a chi-squared distribution with degrees of freedom $3$, 
		and $\xi_4$ is a chi-squared distribution with degrees of freedom $4$. 
		Then $\xi$ has mean and covariance 
		$$\mu=\begin{bmatrix}
			\frac{1}{2}\\2\\3\\4
		\end{bmatrix}, ~\Lambda=\begin{bmatrix}
			\frac{1}{20} & 0 & 0  & 0\\0 & 2 & 0 & 0\\0 & 0 & 6 &0\\0&0&0&8
		\end{bmatrix}.$$
		It is easy to verify that $f$, $-u_1$, and $-u_2$ are all SOS-convex polynomials. 
		By Theorem~\ref{quantileproof}, we computed the initial set size with $N = 300$ samples and the probability violation
		\[
		\Gamma_1 \,=\, 8.1934,\quad
		p_{vio}(x^*(\Gamma_1))\,=\, 0.0000.
		\]
		It took $36.2564$ seconds for Algorithm~\ref{iterativegorithm} to terminate after $23$ loops,
		where Algorithm ~\ref{Sdp_Relax_algorithm_for_SOS-convex_CCO} stops at the initial loop in each inner iteration.
		In the terminating loop, we get 
		\[\begin{array}{c}
			\Gamma^* \,=\, 0.2918,\quad  p_{vio}(x^*) \,=\,0.1500,\\
			f^{*} \,= \, -8.2948,\quad
			x^* \,=\, (0.5030,\, -1.7362,\, 0.0185,\, -0.0240)^T.
		\end{array}\]
	}
\end{example}

Next, we give an example of individual CCO problems with an application background in portfolio selection problems.
\begin{example} \label{var_portfoilio_optimizatio}{\rm\textbf{(VaR Portfolio Optimization)} 
		An investor intends to invest in $4$ risky assets. 
		The goal is to get the minimal loss level such that the probability of a bigger loss 
		does not exceed a preferred risk level $\epsilon$.
		Value-at-Risk (VaR) portfolio optimization can be used to model this problem. 
		It has the form
		\begin{eqnarray*}
			\left\{\begin{array}{cl}
				\min\limits_{(x_0,x)} & x_0  \\
				\st & \mathbb{P}\Big\{\xi:~ x_0\geq -\sum\limits_{i=1}^4 x_ir_i(\xi)\Big\} \geq 1-\epsilon,\\
				& \sum\limits_{i=1}^4x_i = 1,\, x_0\in\re,\\
				& x = (x_1,x_2,x_3,x_4)\in\re^4_+,
			\end{array}\right.
		\end{eqnarray*}
		where $\epsilon$ is a preferred risk level, and each function $r_i(\xi)$ 
		stands for the return rate function of the $i$-th asset. Assume
		$\xi = (\xi_1,\xi_2,\xi_3)$ and
		\[\begin{array}{ll}
			r_1(\xi)\,=\, 0.5+\xi_1^2-\xi_2^2\xi_3^2+\xi_1^4,  
			&r_2(\xi)\,=\, -1+\xi_2^2+\xi_2^4-\xi_1^2\xi_3^2,  \\
			r_3(\xi)\, =\, 0.8+\xi_3^2-\xi_1\xi_2+\xi_3^4,
			&r_4(\xi)\,=\, 0.5+\xi_3-\xi_1\xi_2^2\xi_3+\xi_1^2\xi_3^2.
		\end{array}\]
		In this application, we consider $\xi_1,\xi_2$ and $\xi_3$ are independently distributed random variables.
		Suppose $\xi_1$ follows the beta distribution with shape parameters $\bar{\alpha}=\bar{\beta}=4$, 
		$\xi_2$ follows the log-normal distribution with parameters $\bar{\mu}=0, \bar{\sigma}=1$ 
		and $\xi_3$ follows the log-normal distribution with parameters $\bar{\mu}=-1,\bar{\sigma}=1$.  Then $\xi$ has the mean vector and covariance matrix
		\[ 
		\mu \,=\, \begin{bmatrix}1/2 \\ \sqrt{e}\\ 1/\sqrt{e}\end{bmatrix}, 
		\quad \Lambda \,=\, \begin{bmatrix} \frac{1}{36} & 0 & 0 \\
			0 & e^2-e & 0\\ 0 & 0 & 1-\frac{1}{e}
		\end{bmatrix}.
		\]
		For a given uncertainty set and for different risk levels, i.e., $\epsilon\,\in\,\{0.05,\, 0.20,\, 0.35\}$, 
		we apply Algorithm~\ref{sdprelaxalgorithm} to solve the robust approximation problem (\ref{eq:robust_approximation}).
		The computational results are reported in Table~\ref{porfolioapplication},
		where $\Gamma_1$ denotes the initial set size implied by Theorem~\ref{iterativegorithm},
		$f_I$ denotes the optimal value of \reff{eq:robust_approximation} with the initial set size.
		\begin{table}[ht] 
			\begin{center}
				\caption{Numerical performance for Example \ref{var_portfoilio_optimizatio} \label{porfolioapplication}}
				\begin{tabular}{cccc}	
					\toprule
					$\epsilon$   & $0.05$ & $0.20$ &  $0.35$ \\
					\midrule
					$\Gamma_1$ & $8.6725$ & $3.9047$ & $3.7130$ \\[3pt]
					$\Gamma^{*}$& $0.5703$ & $0.31374$ & $0.1191$ \\[3pt]
					$p_{vio}(x^*(\Gamma_1))$ & $3.0600\cdot 10^{-4}$ & $0.0011$ & $0.0031$ \\[3pt]
					$p_{vio}(x^*)$& $0.0500$ & $0.2000$ & $0.3500$ \\[3pt]
					$f_I$ & $-0.5340$ & $-0.5364$ & $-0.5365$\\[3pt]
					$f^{*}$& $-0.5598$ & $-0.6642$ & $-0.8127$ \\[3pt]
					$x^{*}$ 
					&  $\begin{bmatrix}0.3909\\ 0.0751\\0.3515\\0.1826\end{bmatrix}$ 
					&$\begin{bmatrix}0.1417\\ 0.0788\\5.3332\cdot 10^{-9}\\0.7795\end{bmatrix}$
					&$\begin{bmatrix}3.3121\cdot10^{-9}\\0.1523\\4.9504\cdot 10^{-9}\\ 0.8477\end{bmatrix}$\\[3pt]
					Loop & $18$ & $20$ & $21$ \\[3pt]
					Time(seconds) &$17.2824$ & $18.6632$  & $19.2371$ \\[3pt]
					\bottomrule
				\end{tabular}
		\end{center}
	\end{table}
}
\end{example}

\section{Conclusion and Discussions}\label{sc:conclusion}
This paper focuses on individual CCO problems of polynomials
with the chance constraint being affine in decision variables.
A robust approximation method is proposed to transform such  polynomially  perturbed CCO
problems into linear conic optimization with nonnegative polynomial cones.
When the objective function is linear and the constraining set $X$ has a semidefinite representation or the objective function is SOS-convex and defining functions for $X$ are all SOS-concave, semidefinite relaxation algorithms, i.e., Algorithm \ref{sdprelaxalgorithm} and Algorithm \ref{Sdp_Relax_algorithm_for_SOS-convex_CCO} are proposed to solve these robust approximations globally. 
Convergent analysis of these algorithms is carried out.
In addition, discussions are made to construct efficient uncertainty sets of 
the robust approximation. A heuristic method is introduced to find good set sizes.
Numerical examples, as well as an application for portfolio optimization, 
are given to show the efficiency of our approach.

There is promising future work for this paper.
For instance, can we extend this method for joint CCO problems with polynomial perturbation?
How do we solve more general polynomial individual CCO problems without convex assumptions?
How to analyze the gap between the robust approximation and the original problem?
These questions motivate our ongoing exploration.


\begin{thebibliography}{60}
	\setlength{\itemsep}{1.5pt}
	\bibitem{AHM14}
	Ahmed, S.:
	Convex relaxations of chance constrained optimization problems. 
	\emph{Optim. Lett.}, 8, 1–-12 (2014). 
	
	\bibitem{ABN00}
	Ben-Tal, A., Nemirovski, A.: 
	Robust solutions of linear programming problems contaminated with uncertain data. 
	\emph{Math. Program.}, 88, 411--424 (2000)
	
	
	\bibitem{DBE09}
	Bertsekas, D.: 
	\emph{Convex Optimization Theory}. 
	Athena Scientific, Nashua (2009)
	
	\bibitem{DBE18}
	Bertsimas, D., Gupta, V., Kallus, N.: 
	Data-driven robust optimization. 
	\emph{Math. Program.}, 167, 235--292 (2018)
	
	\bibitem{DBE21}
	Bertsimas, D., Hertog, D., Pauphilet, J.: 
	Probabilistic guarantees in robust optimization. 
	\emph{SIAM J. Optim.}, 31, 2893--292 (2021)
	
	\bibitem{DBE04}
	Bertsimas, D., Sim. M.: 
	The price of robustness.  
	\emph{Oper. Res.}, 52, 35--53 (2004)
	
	\bibitem{DBI14}
	Bienstock, D., Chertkov, M., Harnett, S.:  
	Chance-constrained optimal power flow: Risk-aware network control under uncertainty. 
	\emph{SIAM Rev.}, 56, 461--495(2014)
	
	\bibitem{GCC06}
	Calafiore, G.C., Campi, M.C.: 
	The scenario approach to robust control design. 
	\emph{IEEE Trans. on Autom. Control}, 51, 742--753 (2006)
	
	\bibitem{YCA20}
	Cao, Y., Victor, M.: 
	A sigmoidal approximation for chance-constrained nonlinear programs. 
	\emph{Preprint} at \url{https://arxiv.org/abs/2004.02402} (2020)
	
	\bibitem{CHA65}
	Charnes, A., Cooper, W.~W., Symonds, G.H.: 
	Cost horizons and certainty equivalents: an approach to stochastic programming of heating oil. 
	\emph{Manag. Sci.} 4, 235-–263 (1958)
	
	\bibitem{CHE19}
	Cheng, J., Gicquel, C., Lisser, A.: 
	Partial sample average approximation method for chance constrained problems. 
	\emph{Optim. Lett.}, 13, 657--672 (2019)
	
	
	\bibitem{RCU05}
	Curto, R., Fialkow, L.: 
	Truncated $K$-moment problems in several variables. 
	\emph{J. Operator Theory}, 54, 189-226 (2005)
	
	
	\bibitem{CDE12}
	Deng, C., Yang, L.: 
	Sample average approximation method for chance constrained stochastic programming in transportation model of emergency management. 
	\emph{Syst. Eng. Procedia.}, 5, 137--143 (2012)
	
	
	\bibitem{DEN21}
	Dentcheva, D.: 
	Optimization models with probabilistic constraints. 
	In: Shapiro, A., Dentcheva, D., Ruszczy{\'e}ski, A. (eds) 
	\emph{Lectures on Stochastic Programming: Modeling and Theory}, Third Edition, 
	pp. 81-149. SIAM, Philadelphia (2021)
	
	
	\bibitem{FEN11}
	Feng, C., Dabbene, F., Lagoa, C. M.: 
	A kinship function approach to robust and probabilistic optimization under polynomial uncertainty.  
	\emph{IEEE Trans. on Autom. Control}, 56, 1509--1523, (2011)
	
	\bibitem{SFE14}
	Feng, C., Dabbene, F., Lagoa, C.M.: 
	A smoothing function approach to joint chance-constrained programs. 
	\emph{J. Optim. Theory. Appl.}, 163, 181-199 (2014) 
	
	\bibitem{GUZ16}
	Guzman, Y., Matthews, L., Floudas, C.: 
	New a priori and a posteriori probabilistic bounds for robust counterpart optimization: I. Unknown probability distributions. 
	\emph{Comput. Chem. Engrg.}, 84, 568-598 (2016).
	
	\bibitem{GUZ17}
	Guzman, Y., Matthews, L., Floudas, C.: 
	New a priori and a posteriori probabilistic bounds for robust counterpart optimization: II. A priori bounds for known symmetric and asymmetric probability distributions.
	\emph{Comput. Chem. Engrg.}, 101, 279--311 (2017).
	
	\bibitem{GUZ17a}
	Guzman, Y., Matthews, L., Floudas, C.: 
	New a priori and a posteriori probabilistic bounds for robust counterpart optimization: III. Exact and near-exact a posteriori expressions for known probability distributions. \emph{Comput. Chem. Engrg.}, 103,116--143 (2017).
	
	\bibitem{HEL10}
	Helton, J.W., Nie, J.: 
	Semidefinite representation of convex sets. 
	\emph{Math. Program.}, 122, 21-64 (2010)
	
	\bibitem{HEN09}
	Henrion, D., Lasserre, J., L{\"o}fberg, J.: 
	GloptiPoly 3: moments, optimization and semidefinite programming. 
	\emph{Optim. Method Softw.}, 24, 761-779 (2009)
	
	
	
	\bibitem{LJH11}
	Hong, L., Yang, Y., Zhang, L.: 
	Sequential convex approximations to joint chance constrained programs: A Monte Carlo approach. 
	\emph{Oper. Res.}, 59, 617-630 (2011)
	
	\bibitem{HON21}
	Hong, L., Huang, Z., Lam, H.: 
	Learning-based robust optimization: Procedures and statistical guarantees. 
	\emph{Manag. Sci.}, 67, 3447-3467 (2021)
	
	\bibitem{RKA21}
	Kannan, R., Luedtke, J.R.: 
	A stochastic approximation method for approximating the efficient frontier of chance-constrained nonlinear programs. 
	\emph{Math. Prog. Comp.}, 13, 705-751 (2021) 
	
	
	\bibitem{LAG05}
	Lagoa, C.M., Li, X., Sznaier, M.: 
	Probabilistically constrained linear programs and risk-adjusted controller design. 
	\emph{SIAM J. Optim.}, 15, 938-–951 (2005).
	
	
	\bibitem{LJB01}
	Lasserre, J.: 
	Global optimization with polynomials and the problem of moments. 
	\emph{SIAM J. Optim.} 11, 796--817 (2001)
	
	\bibitem{LJB09}
	Lasserre, J.: 
	Convexity in semialgebraic geometry and polynomial optimization. 
	\emph{SIAM J. Optim.} 19, 1995-2014 (2009)
	
	
	\bibitem{LJB17}
	Lasserre, J. B.: 
	Representation of chance-constraints with strong asymptotic guarantees. 
	\emph{Control Syst. Lett.}, 1, 50-55 (2017).
	
	\bibitem{MLA09}
	Laurent, M.: 
	Sums of squares, moment matrices and optimization over polynomials. In: Putinar, M., Sullivant, S. (eds.) \emph{Emerging Applications of Algebraic Geometry}, 
	pp. 157-270. Springer, New York (2009)
	
	
	
	
	\bibitem{ZLI12}
	Li, Z., Tang, Q., Floudas, C.: 
	A comparative theoretical and computational study on robust counterpart optimization: II. Probabilistic guarantees on constraint satisfaction.  
	\emph{Ind. Eng. Chem. Res.}, 51, 6769-6788 (2012)
	
	\bibitem{ZLI14}
	Li, Z.,  Tang, Q., Floudas, C.: 
	A comparative theoretical and computational study on robust counterpart optimization: III. Improving the quality of robust solutions. 
	\emph{Ind. Eng. Chem. Res.} 53,  13112-13124 (2014)
	
	\bibitem{JLO05}
	L{\"o}fberg, J.: 
	YALMIP: a toolbox for modeling and optimization in MATLAB. 
	\emph{2004 IEEE International Conference on Robotics and Automation} (IEEE Cat. No.04CH37508). 284-289 (2005)
	
	\bibitem{JLU08}
	Luedtke,  J., Ahmed, S.: A sample approximation approach for optimization with probabilistic constraints.  
	\emph{SIAM J. Optim.},  19, 674-699 (2008)
	
	\bibitem{JAS15}
	Jasour, A. M.; Aybat, N. S.; Lagoa, C. M.: 
	Semidefinite programming for chance constrained optimization over semialgebraic sets. \emph{SIAM J. Optim.} 25, 1411–-1440 (2015)
	
	\bibitem{ANE06}
	Nemirovski, A., Shapiro, A.: 
	Scenario approximations of chance constraints. 
	In: Calafiore, G., Dabbene, F. (eds.) 
	\emph{Probabilistic and Randomized Methods for Design under Uncertainty}, 
	pp. 3-47. Springer, London (2006)
	
	\bibitem{ANE07}
	Nemirovski, A, Shapiro, A.:  
	Convex approximations of chance constrained programs. 
	\emph{SIAM J. Optim.} 17, 969-996 (2007)
	
	
	\bibitem{NIE14}
	Nie, J.: 
	The ${\mathcal {A}}$-truncated $K$-moment problem. 
	\emph{Found. Comput. Math.}, 14, 1243-1276 (2014)
	
	\bibitem{JNI14}
	Nie, J.:  
	Optimality conditions and finite convergence of Lasserre’s hierarchy. 
	\emph{Math. Program.}, 146, 97-121 (2014)
	
	\bibitem{NIE15}
	Nie, J.: 
	Linear optimization with cones of moments and nonnegative polynomials.  
	\emph{Math. Program.}, 153, 247-274 (2015)
	
	\bibitem{NIE13}
	Nie, J.: Certifying convergence of Lasserre’s hierarchy via flat truncation.  
	\emph{Math. Program.}, 142, 485-510 (2013)
	
	\bibitem{NIE09}
	Nie, J., Ranestad, K.: 
	Algebraic degree of polynomial optimization. 
	\emph{SIAM J. Optim.}, 20, 485-502 (2009)
	
	\bibitem{NIE07}
	Nie, J., Schweighofer, M.: 
	\emph{On the complexity of Putinar's Positivstellensatz.} 
	\emph{J. Complexity.}, 23, 135-150 (2007)
	
	\bibitem{Niebook}
	Nie, J., 
	\emph{Moment and Polynomial Optimization}. 
	SIAM, 2023.
	
	\bibitem{BKP09}
	Pagnoncell, B.K., Ahmed, S., Shapiro, A.: 
	Sample average approximation method for chance constrained programming: theory and applications. 
	\emph{J. Optim. Theory. Appl.} 142, 399-416 (2009)
	
	\bibitem{PRE03}
	Pr\'{e}kopa, A.: Probabilistic programming. In:  
	Ruszczynski, A., Shapiro, A. (eds.)  
	\emph{Handbooks in Operations Research and Management Science}, 
	vol. 10, pp.  267-351. Elsevier  (2003)
	
	\bibitem{PUM93}
	Putin M.: 
	Positive polynomials on compact semi-algebraic sets. 
	\emph{Indiana Univ. Math. J.}, 42, 969-984 (1993)
	
	
	\bibitem{ROC00}
	Rockafellar, R., Uryasev, S.: 
	Optimization of conditional value-at-risk. 
	\emph{Journal of risk} 2, 21-42 (2000)
	
	\bibitem{STU99}
	Sturm, J.: 
	Using SeDuMi 1.02, a MATLAB toolbox for optimization over symmetric cones. 
	\emph{Optim. Method Softw.}, 11, 625-653 (1999)
	
	\bibitem{YYU17}
	Yuan, Y., Li, Z., Huang, B.: 
	Robust optimization approximation for joint chance constrained optimization problem. 
	\emph{J. Glob. Optim.}, 67, 805-827 (2017)
	
\end{thebibliography}
\end{document}